\documentclass[11pt]{article}

\usepackage[T1]{fontenc}
\usepackage[utf8]{inputenc}

\usepackage{amsmath}
\usepackage{amsthm}
\usepackage{mathtools}
\usepackage{amsfonts}
\usepackage{url}
\usepackage{microtype}

\usepackage{newtxtext,newtxmath}

\usepackage[top=1.1in,bottom=1.1in,left=1.2in,right=1.2in]{geometry}

\usepackage[font=small]{caption}
\usepackage[font=footnotesize]{subcaption}

\usepackage{graphicx}
\usepackage{setspace}
\usepackage{xcolor}
\usepackage{color}

\usepackage[scaled=.90]{helvet}
\usepackage{titlesec}
\titleformat{\section}{\normalfont\sffamily\Large\bfseries}
  {\thesection}{1em}{}
\titleformat{\subsection}{\normalfont\sffamily\large\bfseries}
  {\thesubsection}{1em}{}

\titleformat{\subsubsection}{\normalfont\sffamily\bfseries}
  {\thesubsubsection}{1em}{}

\usepackage[numbers,sort&compress]{natbib}

\renewcommand{\phi}{\varphi}
\newcommand{\phiexact}{\phi_{\text{exact}}}

\DeclareMathOperator{\Tr}{Tr}

\newcommand{\exx}{\mathbf{e}_1}
\newcommand{\eyy}{\mathbf{e}_2}
\newcommand{\ezz}{\mathbf{e}_3}

\usepackage{bm}

\newcommand{\mtx}[1]{\bm{\mathsf{#1}}}
\newcommand{\entry}[1]{{\mathsf{#1}}}

\newcommand{\surfdiv}{\nabla_\Gamma \cdot}
\newcommand{\surfgrad}{\nabla_\Gamma}
\newcommand{\surflap}{\Delta_\Gamma}

\newcommand{\bbR}{\mathbb R}

\newcommand{\bn}{\bm n}
\newcommand{\br}{\bm r}
\newcommand{\bx}{\bm x}
\newcommand{\bs}{\bm s}
\newcommand{\bL}{\bm L}

\newcommand{\bB}{\bm B}
\newcommand{\bE}{\bm E}
\newcommand{\bF}{\bm F}
\newcommand{\bH}{\bm H}
\newcommand{\bJ}{\bm J}
\newcommand{\bK}{\bm K}
\newcommand{\bV}{\bm V}

\newcommand{\cI}{\mathcal I}
\newcommand{\cM}{\mathcal M}
\newcommand{\cP}{\mathcal P}
\newcommand{\cA}{\mathcal A}
\newcommand{\cS}{\mathcal S}
\newcommand{\cD}{\mathcal D}
\newcommand{\cW}{\mathcal W}

\newcommand{\npol}{n_{\text{pol}}}
\newcommand{\nquad}{n_{\text{quad}}}
\newcommand{\ntri}{n_{\text{tri}}}
\newcommand{\npts}{n_{\text{pts}}}

\newtheorem{theorem}{Theorem}[section]
\newtheorem{lemma}[theorem]{Lemma}

\newtheorem{corollary}[theorem]{Corollary}

\title{\bf\sffamily Second-kind integral equations for the Laplace-Beltrami
  problem on  surfaces in three dimensions}

\numberwithin{equation}{section}

\author{Michael O'Neil\footnote{Courant Institute,  New York University, New
    York, NY. Email: {\tt oneil@cims.nyu.edu}. Research supported in
  part by the Office of Naval Research under Award N00014-15-1-2669.}}

\date{\today}

\begin{document}

\maketitle 
\begin{abstract}
  The Laplace-Beltrami problem $\surflap \psi = f$ has several
  applications in mathematical physics, differential geometry, machine
  learning, and topology. In this work, we present novel second-kind
  integral equations for its solution 
  which obviate the need for constructing a   suitable parametrix to
  approximate the in-surface Green's function. The resulting integral
  equations are well-conditioned and compatible with standard fast
  multipole methods and iterative linear algebraic solvers, as well as
  more modern fast direct solvers. 
  Using
   layer-potential identities known as Calder\'on projectors, the
   Laplace-Beltrami operator can be pre-conditioned from the left
   and/or right to obtain second-kind integral equations. We
   demonstrate the accuracy and stability of the scheme in several
   numerical examples along surfaces described by curvilinear triangles.
\end{abstract}

%\onehalfspacing

\section{Introduction}

For a smooth, closed surface $\Gamma$ embedded in $\bbR^3$, the
Laplace-Beltrami problem
\begin{equation}\label{eq_lapbel}
\surflap \psi = f
\end{equation}
has numerous applications in partial differential
equations~\cite{EpGr,veerapaneni_2011,schwartz_2005},
topology~\cite{reuter_2010}
and differential geometry~\cite{mckean_1967},
shape optimization~\cite{sokolowski_1992},
computer graphics~\cite{chao_2010}, and even machine
learning~\cite{weiss_2009,belkin_2006}.
The main application area 
in mind for this work is that of electromagnetics, where it is often
 useful to partition tangential vector fields (e.g. electric
current)
on surfaces of
arbitrary genus into their Hodge decomposition~\cite{EpGr,EpGrOn}.
The Hodge decomposition is the generalization of the standard Helmholtz
decomposition to multiply-connected domains, and allows for a
component of the vector field which is divergence \emph{and} curl
free.
In particular, any tangential vector field $\bF$ along a smooth,
closed surface $\Gamma$ can be
written as
\begin{equation}
\bF = \surfgrad \alpha  + \bn \times 
\surfgrad \beta  + \bH,
\end{equation}
where $\surfgrad$ denotes the surface gradient, $\alpha$, $\beta$
are functions defined along~$\Gamma$, and $\bn$ is the unit normal
vector.
(In what follows,
we will restrict our 
discussion and applications to smooth vector fields $\bF$.)
The surface divergence
is denoted by $\surfdiv$, and the \emph{surface curl} is given by
$\surfdiv (\bn \times)$ so that
\begin{equation}
\surfdiv \bF = \surflap \alpha, \qquad \surfdiv (\bn \times \bF) =
-\surflap \beta.
\end{equation}
The vector field $\bH$ is the harmonic component of $\bF$, and has
zero divergence and curl:
\begin{equation}
\surfdiv \bH = 0, \qquad \surfdiv (\bn \times \bH) = 0.
\end{equation}
The existence of non-trivial vector 
fields $\bH$ depends on the genus of $\Gamma$:
the dimension of this harmonic subspace is $2p$, with $p$
denoting the genus of~$\Gamma$.  
The exact usefulness of this decomposition with
regard to problems in electromagnetics is discussed in detail
in~\cite{EpGr,EpGrOn,EpGrOn2,dai_2014}. To summarize one particular
application, electromagnetic fields
in the exterior of a perfect electric conductor (PEC)
with smooth boundary $\Gamma$
can be
represented using what are known as \emph{generalized Debye
  sources}, denoted below by $r$, $q$.
After sufficient scaling,
the time-harmonic Maxwell equations with wavenumber $k$ in a
source and current free region can be written as~\cite{papas}:
\begin{equation}
  \begin{aligned}\label{eq_maxwell}
  \nabla \times \bE &= ik \bB, &\qquad \nabla \times \bB &= -ik \bE,\\
  \nabla \cdot \bE &= 0, &   \nabla \cdot \bB &= 0.
  \end{aligned}
\end{equation}
With the above standardization, the 
electric and magnetic fields in the exterior of a PEC can be
represented
as
\begin{equation}\label{eq_gendebye}
\begin{aligned}
  \bE &= ik \cS_k \bJ - \nabla \cS_k r - \nabla \times \cS_k \bK, \\
  \bB &= ik \cS_k \bK - \nabla \cS_k q + \nabla \times \cS_k \bJ,
\end{aligned}
\end{equation}
where $k$ is the wavenumber of the fields
 and $\cS_k$ is the single-layer Helmholtz
operator:
\begin{equation}
\cS_k f (\bx) =  \int_\Gamma \frac{e^{ik|\bx-\bx'|}}{4\pi |\bx-\bx'|} 
\, f(\bx') \, da(\bx') .
\end{equation} 
The surface vector fields $\bJ$ and $\bK$ are defined using $r$, $q$
in Hodge form
so that $\bE$, $\bB$ in~\eqref{eq_gendebye}
automatically satisfy Maxwell's equations:
\begin{equation}
\begin{aligned}
\bJ &=  ik \left(\surfgrad \surflap^{-1} r - \bn \times \surfgrad
  \surflap^{-1} q  \right) + \bJ_H ,\\
\bK &= \bn \times \bJ,
\end{aligned}
\end{equation}
where $\bJ_H$ is a harmonic vector field along $\Gamma$.
Note that the construction of $\bJ$ and $\bK$ requires the application
of $\surflap^{-1}$ (or, equivalently, the solution of $\surflap \psi = r$).

Despite the many applications of the Laplace-Beltrami
problem in widely varying
domains, the literature surrounding its numerical 
solution is rather limited.
Methods which directly discretize the differential operator are
usually based on finite elements or finite
differences~\cite{bonito_2013,
  hansbo_2016,demlow_2007,dziuk_2013,frittelli_2016}.
Alternatively, there are a
few methods~\cite{kropinski_2014_fast,kropinski_2016_integral}
relevant
to problems along surfaces in three dimensions which
re-formulate the problem in integral form using a parametrix of
$\log$-type, which to leading order, is the Green's function for
$\surflap$ on surfaces embedded in three dimensions.
Integral equations on surfaces in
three dimensions is a well-studied field~\cite{kress_2014}, but usually
the kernels involved are Green's functions corresponding to
constant-coefficient PDEs in
three dimensions, and not variable coefficient ones
along two-dimensional surfaces. Outside of
special-case geometries, novel quadrature rules and fast algorithms
would need to be constructed in order to rapidly solve integral
equations on surfaces in three dimensions with logarithmically-singular
kernels.

If one is willing to obtain or construct a signed distance function
(defined in the volume) to the surface~$\Gamma$, then various
level-set methods become options for solving surface PDEs, including
the Laplace-Beltrami problem. This approach was taken
in~\cite{bertalmio2001,greer2006,chen2015,macdonald2009,macdonald2008,ruuth2008} for
2nd-order and 4th-order surface diffusion-type problems, including the
Cahn-Hilliard equation. These schemes are based on using an embedded
finite difference scheme (in the volume) to discretize an extension of
the surface
PDE~(often using formulas such as~\eqref{eq_vollap}, below).
While very general (and very efficient yielding sparse matrices
to invert), and
certainly applicable to a wide range of surface PDEs, the
solvers associated with level set methods and the closest point method
can suffer from ill-conditioning in the presence of adaptive
discretizations (possibly needed in order to resolve sharp geometric
features or high-bandwidth right-hand sides).
Furthermore, methods based on such volumetric extensions inherently
require extending the right-hand side from the surface to the volume,
which is a notoriously difficult task to do with high-order
accuracy~\cite{askham2017}.

With this in mind, in this work we introduce special-purpose,
second-kind boundary
integral equations obtained by applying left- and
right-preconditioners to the Laplace-Beltrami problem that can be used
for its efficient solution. The resulting integral equations rely on
several Calder\'on identities for harmonic layer potentials, and
contain integral operators whose kernel is the Green's function for
the \emph{three-dimensional Laplacian}. These features allow for the
immediate use of fast algorithms, such as fast multipole methods
(FMMs)~\cite{greengard-1997}, and standard quadrature methods for
weakly singular integrals along
surfaces~\cite{bremer_2012c,bremer_2013}.  We will demonstrate the
effectiveness of these integral equations with several numerical
examples.

The paper is organized as follows:
In Section~\ref{sec_intro} we introduce the Laplace-Beltrami operator
and provide some standard relationships with the three-dimensional
Laplacian.
In Section~\ref{sec_integral} we derive novel second-kind integral
equations useful in solving the Laplace-Beltrami problem. In
Section~\ref{sec_surfaces} we describe the numerical discretization of
arbitrary curvilinear surfaces and derive various quantities from 
differential geometry. Section~\ref{sec_numerical} contains
several examples of the resulting solver, and finally in
Section~\ref{sec_conclusions} , we discuss
advantages, drawbacks, and future improvements  of our scheme.

\section{The Laplace-Beltrami operator}
\label{sec_intro}

The Laplace-Beltrami operators, also known as the surface Laplacian,
is the generalization of the Laplace operator to general curvilinear
coordinates. 
In what follows, we will always assume that the surface $\Gamma$ has
at least two continuous derivatives, and in practice, we provide examples
which have higher degrees of smoothness.
For a surface $\Gamma$ parameterized (at least locally) by a
smooth function $\bx: \bbR^2 \to \bbR^3$,
\begin{equation}\label{eq_triparam}
\bx(u,v) = x_1(u,v) \, \exx + x_2(u,v) \, \eyy + x_3(u,v) \, \ezz,
\end{equation}
where $\exx$, $\eyy$, $\ezz$ is the elementary orthonormal basis for
$\bbR^3$, the metric tensor is given by
\begin{equation}
g = \begin{pmatrix}
\bx_u  \cdot \bx_u & 
\bx_u  \cdot \bx_v \\
\bx_v  \cdot \bx_u & 
\bx_v  \cdot \bx_v
\end{pmatrix}.
\end{equation}
The determinant of the metric tensor will be denoted as $\det g =
|g|$, and the components of $g$ as:
\begin{equation}
\begin{aligned}
 g_{uu} &= \bx_u \cdot \bx_u, &\qquad g_{uv} &= \bx_u \cdot \bx_v, \\
 g_{vu} &= \bx_v \cdot \bx_u, &\quad g_{vv} &= \bx_v \cdot \bx_v.
\end{aligned}
\end{equation}
Here, we use the abbreviations $\bx_u = \partial \bx/\partial u$ and 
$\bx_v = \partial \bx/\partial v$.
The unit \emph{outward} normal vector is given by
\begin{equation}
\bn(\bx) = \frac{\bx_u \times \bx_v}{|\bx_u \times \bx_v|}
\end{equation}
so that the vectors $\bx_u$, $\bx_v$, and $\bn$ form a right-handed
system of coordinates.
The surface divergence and gradient are then~\cite{frankel,nedelec}:
\begin{equation}
  \begin{aligned}
    \surfdiv \bF &= \frac{1}{\sqrt{ |g|}} \left(
      \frac{\partial}{\partial u} \left( \sqrt{|g|} \, F^u  \right)+ 
      \frac{\partial}{\partial v} \left( \sqrt{|g|} \, F^v \right) 
    \right), \\
    \surfgrad \psi 
    &=  \left( g^{uu}
      \frac{\partial \psi}{\partial  u}
      + g^{uv} \frac{\partial \psi}{\partial v} \right)  \bx_u + 
    \left( g^{vu} \frac{\partial \psi}{\partial u} +
      g^{vv} \frac{\partial \psi}{\partial v}
        \right) \bx_v,    
  \end{aligned}
\end{equation}
where $\psi = \psi(u,v)$ is a scalar function along $\Gamma$,
$\bF = \bF(u,v)$ is a tangential vector field defined with respect to 
the tangent vectors $\bx_u$ and $\bx_v$:
\begin{equation}
  \bF(u,v) =  F^u(u,v) \, \bx_u + F^v(u,v) \, \bx_v ,
\end{equation}
and the coefficients~$g^{ij}$ are the components of the inverse
of~$g$:
\begin{equation}
  g^{-1} =  \begin{pmatrix}
g^{uu} &  g^{uv} \\
g^{vu} &  g^{vv}
\end{pmatrix}.
\end{equation}
If the surface~$\Gamma$ is closed, 
the surface gradient and divergence yield the integration by parts
formula:
\begin{equation}
\int_\Gamma \surfgrad \psi \cdot \bF \, da = -\int_\Gamma \psi \, 
\surfdiv \bF \, da, 
\end{equation}
where $da = \sqrt{|g|} \, du \, dv$ is the surface-area differential.
The Laplace-Beltrami operator is then defined as~$\surflap =
\surfdiv\surfgrad$, and given explicitly as:
\begin{equation}
  \surflap \psi = \frac{1}{\sqrt{|g|}} \left( 
    \frac{\partial}{\partial u} \sqrt{|g|} 
    \left( g^{uu} \, \frac{\partial\psi }{\partial u} + 
    g^{uv} \, \frac{\partial\psi }{\partial v} \right) + 
  \frac{\partial}{\partial v} \sqrt{|g|} \left( 
    g^{vu} \, \frac{\partial\psi }{\partial u} + 
    g^{vv} \, \frac{\partial\psi }{\partial v} \right)
\right).
\end{equation}
From the previous definition of the Laplace-Beltrami operator, it is
clear that in general, the problem of solving $\surflap \psi = f$
requires solving a second-order variable-coefficient PDE in the
variables $u$, $v$. Solutions of this variable coefficient PDE can be
obtained via pseudo-spectral methods if a tractable parameterization
of the boundary can be obtained~\cite{imbertgerard_2017} or by finite
differences or finite element methods~\cite{bonito_2013,dziuk_2013}. In each of
these cases, complicated geometry plays a key role.
The goal of this paper will be to avoid designing methods for variable
coefficient PDEs along surfaces, and instead obtain an equivalent
well-conditioned linear integral equation compatible with existing
fast algorithms that can be used to solve the Laplace-Beltrami problem.

Alternatively, for functions defined not just along~$\Gamma$, but in a
neighborhood of~$\Gamma$, the surface Laplacian can be written in terms of the
standard three-dimensional Laplacian plus correction terms
which remove differential contributions in the normal direction.
See, for example,~\cite{nedelec} for a derivation of the
following lemma.

\begin{lemma}\label{lem_lap}
  For an object~$\Omega$ in three dimensions with smooth boundary
  given by~$\Gamma$ and a twice differentiable function
  $\psi = \psi(\bx)$ defined in a neighborhood of $\Gamma$, we have
  that
\begin{equation}\label{eq_vollap}
  \Delta \psi = \surflap \psi + 2H 
  \frac{\partial \psi}{\partial n} +
  \frac{\partial^2 \psi}{\partial n^2},
\end{equation}
where $H$ is the signed mean curvature,
\begin{equation}
H = \frac{\nabla \cdot \bn }{2},
\end{equation}
and
$\partial/\partial n$ denotes differentiation in the direction normal
(outward) to the boundary of $\Omega$:
\begin{equation}
\frac{\partial \psi}{\partial n} = \bn \cdot \nabla \psi.
\end{equation}
\end{lemma}

Note that in the previous lemma, we have defined the mean curvature
$H$ so that the unit sphere has $H = 1$, assuming that the outward
normal is given by $\bn = \br$.
The following corollary addresses the case in which
the function $\psi$ is harmonic.
\begin{corollary}
If the function $\psi$ defined in Lemma~\ref{lem_lap} is also harmonic
in the neighborhood of $\Gamma$, then we have that
\begin{equation}
\surflap \psi = -2H  \frac{\partial \psi}{\partial n} - 
  \frac{\partial^2 \psi}{\partial n^2}.
\end{equation}
\end{corollary}

Finally, we conclude this section with a brief discussion
of the invertibility of the Laplace-Beltrami operator.
Clearly, $\surflap$ contains a
 nullspace of dimension one, namely that of constant 
functions along~$\Gamma$. Therefore, in order to formulate a
well-posed Laplace-Beltrami problem, we must restrict
the operator, and its inverse, to mean-zero functions (or enforce some
other constraint).
The following lemma is a standard result in Hodge theory, and
discussed in more detail in~\cite{nedelec, EpGr}.

\begin{lemma}
For smooth closed boundaries~$\Gamma$, the Laplace-Beltrami operator
is uniquely invertible as a map from $\cM_0$ to $\cM_0$, the space of
mean-zero functions defined on~$\Gamma$:
\begin{equation}
\cM_0 = \left\{ \psi : \int_\Gamma \psi \, da = 0 \right\}.
\end{equation}
\end{lemma}
That is to say, for a function $f \in \cM_0$, there is a unique
twice-differentiable $\psi \in \cM_0$ such that $\surflap \psi = f$.
In what follows in this section and all subsequent ones, we will
assume that the right hand side of the Laplace-Beltrami equation, $f$,
is a continuous function in $ L^2(\Gamma)$ in order to avoid a
discussion of singular
layer potential densities.
We now have that the
following PDE + constraint is a well-posed system:
\begin{equation}
  \begin{aligned}
    \surflap \psi &= f \qquad \text{on }\Gamma,\\
    \int_\Gamma \psi &= 0.
  \end{aligned}
\end{equation}
Given these facts regarding the Laplace-Beltrami operator, we now turn
to deriving integral equations for its solution.

\section{Integral equation formulations}
\label{sec_integral}

In this section, we derive second-kind integral equations which can be
used to solve the Laplace-Beltrami problem along smooth surfaces in
three dimensions.
These integral equations  rely only on standard layer
potentials, their normal derivatives, and local curvature information
of~$\Gamma$.
The resulting solvers are immediately compatible with
standard fast algorithms and iterative solvers (e.g. FMMs and 
GMRES~\cite{saad-1986}).
The main working tool of the following
derivation will be what are known as \emph{Calder\'on identities},
which we discuss now.

\subsection{Calder\'on projectors}
\label{sec_calderon}

In order to derive well-conditioned second-kind integral equations for
the Laplce-Beltrami equation along surfaces, it is useful to
precondition equation~\eqref{eq_lapbel} (on both the left and the
right)
and invoke what are known as Calder\'on relations (also known as
identities or projections) to rewrite the resulting operator in 
diagonal and compact terms.
To this end, we briefly provide several Calder\'on identities that
will be useful. These identities, and simple proofs using only
standard Green's identities, can be found in~\cite{nedelec}.

Let us denote by $G$ the Green's function for the Laplace equation in
three dimensions:
\begin{equation}
G(\bx-\bx') = \frac{1}{4\pi |\bx - \bx'|}.
  \end{equation}
This Green's function satisfies $\Delta_{\bx} G(\bx-\bx') =
-\delta(\bx-\bx')$, where $\delta$ is the Dirac delta function defined
in the proper sense of distributions~\cite{folland_1995}, and
$\Delta_{\bx}$ denotes the Laplacian applied in the $\bx$ variable.
For $\bx \notin \Gamma$, we next denote by
$\cS$ and $\cD$ the standard single- and double-layer
potential operators 
for Laplace potentials, given by:
\begin{equation}\label{eq_singdoub}
\cS\sigma(\bx) = \int_\Gamma G(\bx-\bx') \, \sigma(\bx')
\, da(\bx'), 
\qquad  \cD\sigma(\bx) = \int_\Gamma \left( \frac{\partial }{\partial n'}
G(\bx-\bx') \right)  \sigma(\bx') \, da(\bx'),
\end{equation}
with $\partial/\partial n' = \bn' \cdot \nabla_{\bx'}$
denoting differentiation in the direction
normal to the surface at $\bx'$.
The functions $\cS \sigma$ and $\cD \sigma$ define smooth functions up
to $\Gamma$; $\cS\sigma$ is continuous across the boundary $\Gamma$,
while $\cD\sigma$ has a jump of size $\sigma$ across $\Gamma$. In
particular, for $\bx \in \Gamma$
\begin{equation}
    \lim_{h \to 0^\pm} \cD\sigma (\bx + h\bn) = \pm \frac{\sigma(\bx)}{2}
    + \cD\sigma(\bx).
\end{equation}
One-sided normal derivatives of the layer potentials $\cS\sigma$ and
$\cD\sigma$ can also be taken. For $\bx \in \Gamma$, we have
\begin{equation}
  \begin{aligned}
  \lim_{h \to 0^\pm} \frac{\partial}{\partial n}
     \cS\sigma (\bx + h\bn) &= \mp \frac{\sigma(\bx)}{2}
    + \cS'\sigma(\bx), \\
  \lim_{h \to 0^\pm} \frac{\partial}{\partial n}
     \cD\sigma (\bx + h\bn) &= \cD'\sigma(\bx),
  \end{aligned}
\end{equation}
where the layer potential operators $\cS'$ and $\cD'$, as maps from
$\Gamma \to \Gamma$, are given by:
\begin{equation}
\cS'\sigma(\bx) = \int_\Gamma \left( \frac{\partial }{\partial n}
G(\bx-\bx') \right)  \sigma(\bx')
\, da(\bx'), 
\qquad  \cD'\sigma(\bx) = \int_\Gamma \left( \frac{\partial^2 }
  {\partial n \, \partial n'}
G(\bx-\bx') \right)  \sigma(\bx') \, da(\bx').
\end{equation}
The function $\cD'\sigma$ is continuous across the interface, while
$\cS\sigma$ has a jump in its normal derivative.
The kernel in $\cD'$ is, however, not
integrable and therefore the integral is interpreted in Hadamard
finite-parts~\cite{nedelec,kress_2014,colton_kress}.
%Going forward,
%for clarity, we will denote by $\cD_\pm$ the exterior/interior limit
%of the double layer potential, and by~$\cD$ the average of the two
%limits, $\cD = (\cD_+ + \cD_-)/2$ (similarly, we use the same notation
%for $\cS'$).

Next, define the matrix $\cP$ of operators from $\Gamma \times \Gamma
\to \Gamma \times \Gamma$ by
\begin{equation}
\cP = \begin{pmatrix}
\cD & \cS \\
-\cD' & -\cS'
\end{pmatrix}.
\end{equation}
This matrix has the projector-like quality (on $\Gamma$) that
$\cP^2 = \cI/4$, where $\cI$ is the identity operator. This
can be proven using a straightforward application
of on-surface Green's identities~\cite{nedelec}.
We then have the four following explicit relationships, known as
\emph{Calder\'on identities}:
\begin{equation}\label{eq_cal}
\begin{aligned}
  \cD \cD - \cS \cD' &= \frac{\cI}{4}, \qquad & \cD \cS &= \cS \cS', \\
\cS' \cS' - \cD' \cS &= 
\frac{\cI}{4},  & \cD' \cD &= \cS' \cD'.
\end{aligned}
\end{equation}
These identities are very useful in constructing preconditioners for
hypersingular integral equations~\cite{contopanagos-2002} and integral
equations for high-frequency acoustic scattering
phenomena~\cite{boubendir-2014}. We will next use these relationships to
construct integral equations for solving the Laplace-Beltrami equation.

\subsection{Laplace-Beltrami integral equations}

Using the Calder\'on identities of the previous section, we can now
derive integral equations of the second-kind which can be used to
solve the Laplace-Betrami problem. For example, letting $\psi = \cS
\sigma$ in equation~\eqref{eq_lapbel} and preconditioning 
on the left yields:
\begin{equation}
\cS \surflap \cS \sigma = \cS f,
\end{equation}
which turns out to be a second-kind integral equation for the
function~$\sigma$. Once $\sigma$
has been found, the original 
function $\psi$ can be computed easily.
To this end, we begin with the following lemma.

\begin{lemma}\label{lem_slaps}
As a map from $\Gamma \to \Gamma$, 
\begin{equation}\label{eq_laps1}
\cS \surflap \cS = -\frac{\cI}{4} +  \cD^2  - 
\cS(\cS'' + \cD') - 2\cS H \cS'  ,
\end{equation}
where $H$ is the mean curvature of $\Gamma$ and the operator
$\cS'':\Gamma \to \Gamma$ is given by
\begin{equation}
  \cS''\sigma(\bx) = \int_\Gamma \left( \frac{\partial^2 }
  {\partial n^2 } G(\bx-\bx') \right)  \sigma(\bx') \, da(\bx').
\end{equation}
\end{lemma}
\begin{proof}
  For any $\sigma$ defined along $\Gamma$,
  the function $\cS \sigma$ is infinitely differentiable off of
$\Gamma$ and is 
continuous across~$\Gamma$.
Furthermore, it defines a
harmonic function and therefore $\Delta \cS \sigma = 0$ along~$\Gamma$.
Using this fact, along with 
Lemma~\ref{lem_lap}, we have that
\begin{equation}\label{eq_lap_s}
\surflap \cS =  - 2H\cS' -   \cS'' .
\end{equation}
The operator $\cS''$ is interpreted in the finite-parts sense, 
as is $\cD'$.
After applying $\cS$ to both sides of equation~\eqref{eq_lap_s},
and adding/subtracting $\cS\cD'$, we have:
\begin{equation}
  \cS \surflap \cS  = -2\cS H \cS' - \cS \left( \cS''
    + \cD'\right) + \cS\cD'.
\end{equation}
Using the Calder\'on identity $\cS\cD' = -\cI/4 + \cD^2$
from~\eqref{eq_cal}, 
we have that
\begin{equation}\label{eq_finaleq}
  \cS \surflap \cS  = -\frac{\cI}{4} -2\cS H \cS' - \cS \left( \cS''
    + \cD'\right) + \cD^2.
\end{equation}
The terms $\cS H \cS'$ and $\cD^2$ are trivially compact due to their
weakly-singular kernels, and the
\emph{difference operator} $\cS'' + \cD'$ can also be shown to be
compact by expressing its kernel as
\begin{equation}
  K = \frac{\partial}{\partial n} \left( \frac{\partial G}{\partial n}
   + \frac{\partial G}{\partial n'} \right).
\end{equation}
As $\bx \to \bx'$, the terms $\partial G/\partial n$ and $\partial
G/\partial n'$ cancel to leading order, and what remains after
applying $\partial/\partial n$ is weakly
singular. See~\cite{rokhlin_1983,kress_1978}
for early uses of this fact.
Therefore,~\eqref{eq_finaleq} gives a second-kind integral operator
formulation of the operator $\cS \surflap \cS$.
\end{proof}

The above identity is straightforward to verify on the unit sphere by
using the diagonal forms of the layer potential operators when applied
to spherical harmonics~\cite{vico_2014}.
Alternatively, a different second-kind formulation exists which
corresponds to a fully right-preconditioned system.
We provide this formulation in the following lemma.
\begin{lemma}
  As a map from $\Gamma\to \Gamma$, 
\begin{equation}\label{eq_laps2}
  \surflap \cS^2 = - \frac{\cI}{4} +\cS'^2 - \left( \cS'' + \cD'
  \right)
  \cS
  - 2H\cS' \cS
\end{equation}
\end{lemma}
\begin{proof}
The proof is similar to that of the previous lemma, and omitted here.
\end{proof}

The operator $\cS^2$ behaves similarly to convolution with
a logarithmic parametrix along $\Gamma$, and for this reason it is not
surprising that the operator in equation~\eqref{eq_laps2} is Fredholm
of the second-kind.
In many respects, integral equation~\eqref{eq_laps1}
and~\eqref{eq_laps2}
behave identically. However, in practice we prefer
equation~\eqref{eq_laps1} due to the fact that the right-hand
side $f$ is smoothed by the application of $\cS$, and therefore fewer
discretization nodes are needed in order to resolve $\cS f$ than are
needed to resolve $f$.

Lastly, we discuss the invertibility of the integral
operators~\eqref{eq_laps1} and~\eqref{eq_laps2}.
As shown in~\cite{imbertgerard_2017, sifuentes_2015},
since the nullspace of the surface Laplacian is exactly
one-dimensional (constant functions) and $\surflap$ is self-adjoint,
letting the operator $\cW$ be defined as
\begin{equation}
\cW \psi = \int_\Gamma \psi \, da
\end{equation}
we have that the integro-differentiable operator $\surflap + \cW$
is of full rank, and uniquely invertible on the space of mean-zero
functions. Therefore, the operators
\begin{equation}
  \cS \left( \surflap + \cW \right) \cS, \qquad\text{and}
  \qquad \left( \surflap + \cW \right) \cS^2
\end{equation}
are Fredholm operators of the second-kind, and uniquely invertible.
The integral expressions for these operators can be obtained from
the expressions in~\eqref{eq_laps1} and~\eqref{eq_laps2} with the
addition of the terms $\cS \cW \cS$ and $\cW \cS^2$, respectively.

\section{Discretization of surfaces and layer potentials}
\label{sec_surfaces}

Often the hardest part of any boundary integral equation method is not
the fast algorithm used to invert the resulting linear system, but
rather obtaining high-order descriptions of the geometry and building
accurate methods with which to compute weakly-singular integrals.  In
this section we give a brief overview of the surface and density
discretizations we
use for subsequent numerical examples for solving the Laplace-Beltrami
problem. Our solver relies on decoupling the surface description from
the discretization of functions along the surface, and it is based on
the one contained in~\cite{bremer_2012c, bremer_2013}.
This allows for high-order discretizations of functions along
low-order surfaces, and vice versa.
While the overall order of the scheme is often limited by the lowest
order of discretization (i.e. the geometry vs. the unknowns), it is
useful for numerical experiments to have control over both orders.

\subsection{Maps of standard triangles}
\label{sec_standard}

We assume that our surface $\Gamma$ is given by a sequence of curvilinear
triangles, each described as a 
smooth map of the standard simplex triangle, denoted by $T_0$, with vertices
$\{(0,0), (1,0), (0,1) \}$.
That is to say, the
$j$th triangle along~$\Gamma$ is parameterized by the map:
\begin{equation}\label{eq_xji}
\bx^j(u,v) = x^j_1(u,v) \, \exx + x^j_2(u,v) \, \eyy + x^j_3(u,v) \, \ezz,
\end{equation}
with $\exx$, $\eyy$, $\ezz$ the standard basis for $\mathbb R^3$.  In
the case of analytically parameterized surfaces, the functions~$x^j_i$
are
known globally, and in other (lower-order)
cases, the surface maps may be provided triangle-by-triangle
as piecewise polynomials:
\begin{equation}\label{eq_tripols}
x_i^j(u,v) = \sum_{m+n \leq p} c^{ji}_{mn} \, u^m \, v^n,
\end{equation}
with~$p$ denoting the \emph{order of the discretization of} $\Gamma$.
In each case, local derivatives of the surface with respect to $u$ and
$v$ can be
computed analytically.
Surface area
elements can be computed using standard differential geometry
formulas, as discussed in Section~\ref{sec_intro}.
The mean curvature at any point $\bx(u,v)$ can be computed
as
\begin{equation}
H = -\frac{1}{2} \Tr\left({II} \cdot {I}^{-1} \right),
\end{equation}
with $I = g$ (the metric tensor)
and $II$ the first and second fundamental forms,
respectively~\cite{docarmo_1976}.
The second fundamental form is defined as
\begin{equation}
  II =
  \begin{pmatrix}
    \bx_{uu} \cdot \bn &     \bx_{uv} \cdot \bn \\
    \bx_{vu} \cdot \bn &      \bx_{vv} \cdot \bn
  \end{pmatrix},
\end{equation}
with
\begin{equation}
\bx_{uu} = \frac{\partial^2 \bx}{\partial u^2}, \qquad
\bx_{uv} = \frac{\partial^2 \bx}{\partial u \, \partial v}, \qquad
\bx_{vv} = \frac{\partial^2 \bx}{\partial v^2},
\end{equation}
and, as usual, $\bn$ is the outward unit normal to the surface.
The second fundamental form contains 
curvature information about the surface.
These partial derivatives can be computed directly from the expression
for each component $x^j_i$ of the map.

Often, surface geometries of engineering interest are generated by
computer aided drafting or engineering software (CAD or CAE), and
described merely by a sequence of image points for each component of
the map; other points on the surface must be obtained by interpolation.
That is to say, given the image point of each node, the coefficients
for the map in~\eqref{eq_tripols} can be determined.
Using standard nodal locations~\cite{gmsh} shown in
Figure~\ref{fig_gmshnodes} it is possible to describe first through
fourth order curvilinear triangles by providing the image of each
node~\cite{gmsh}. Higher order triangles, of course, require more node
locations.
While these nodes may not be optimal in terms of the conditioning of
the resulting interpolation formula, they are suitably stable for
solving for the coefficients in a Koornwinder polynomial basis using
least squares. Namely, the component maps can equally
be expressed as
\begin{equation}
x_i^j(u,v) = \sum_{m+n \leq p} c^{ji}_{mn} \, K_{mn}(u,v),
\end{equation}
where the $K_{mn}$'s are the affine-translated Koornwinder polynomials on
$T_0$~\cite{bremer_2012c, koornwinder_1975} (these polynomials are
two-variable extensions of classical one-variable
orthogonal polynomials).

\begin{figure}[!t]
  \begin{subfigure}[b]{.22\linewidth}
    \centering
    \includegraphics[width=.9\linewidth]{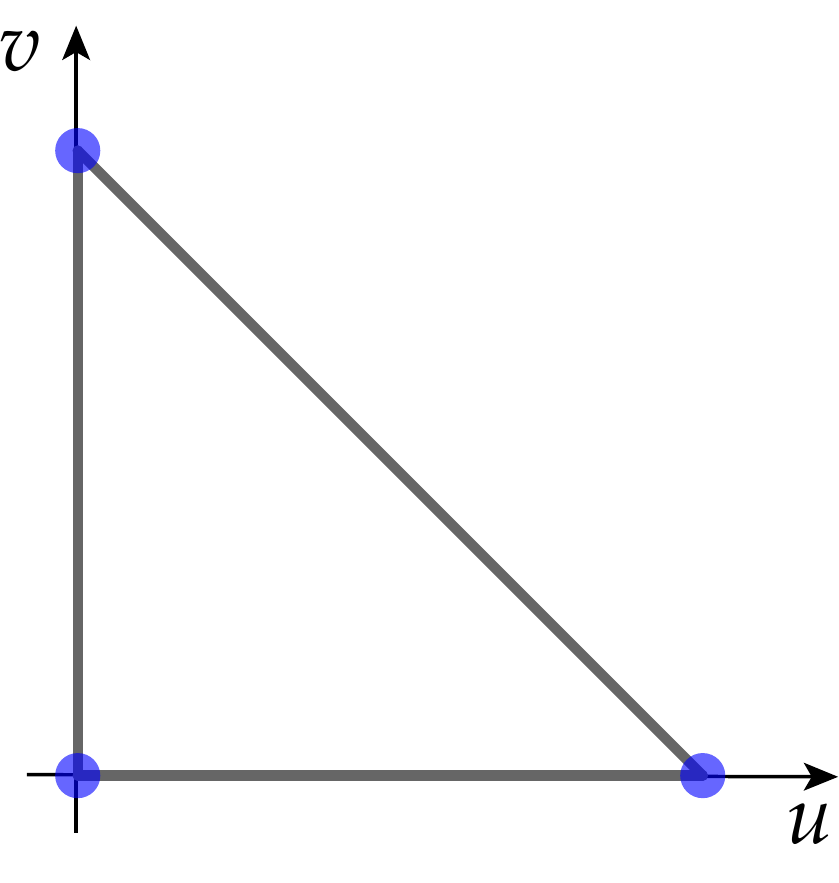}
    \caption{First-order nodes.}
    \label{fig_tri1}
  \end{subfigure}
  \begin{subfigure}[b]{.22\linewidth}
    \centering
    \includegraphics[width=.9\linewidth]{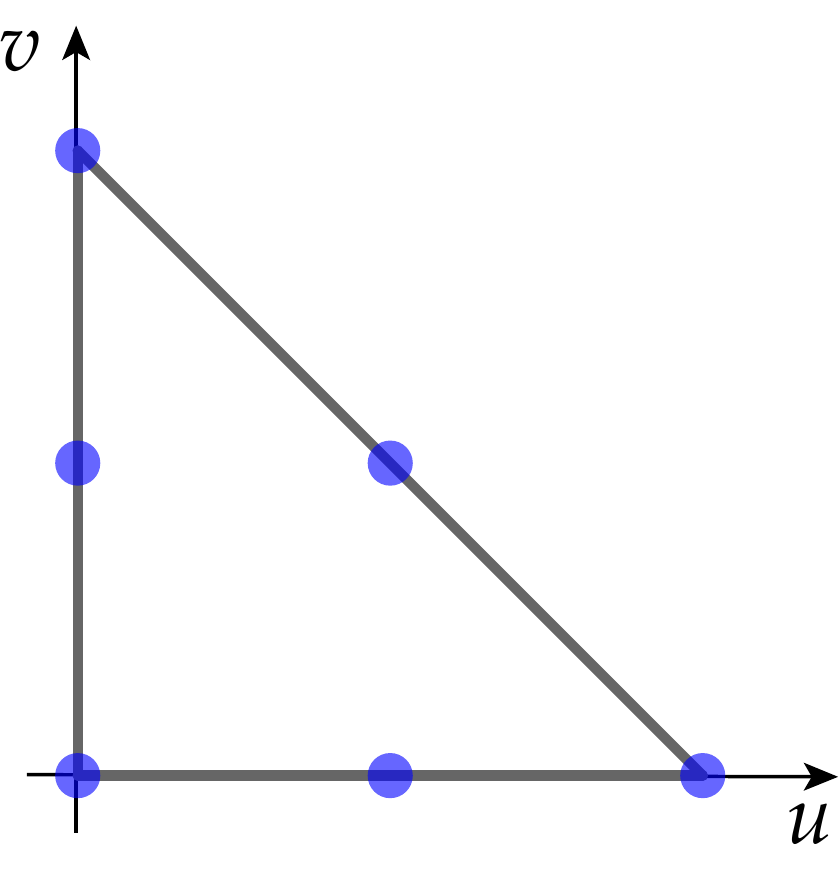}
    \caption{Second-order nodes.}
    \label{fig_tri2}
  \end{subfigure}
  \begin{subfigure}[b]{.22\linewidth}
    \centering
    \includegraphics[width=.9\linewidth]{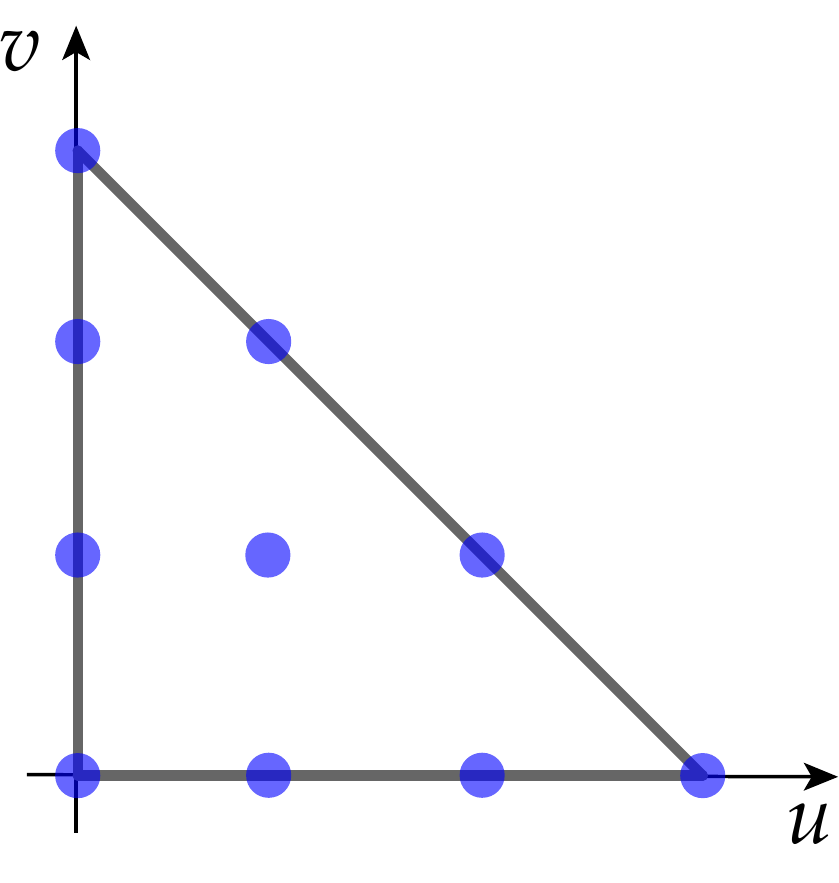}
    \caption{Third-order nodes.}
    \label{fig_tri3}
  \end{subfigure}
  \begin{subfigure}[b]{.22\linewidth}
    \centering
    \includegraphics[width=.9\linewidth]{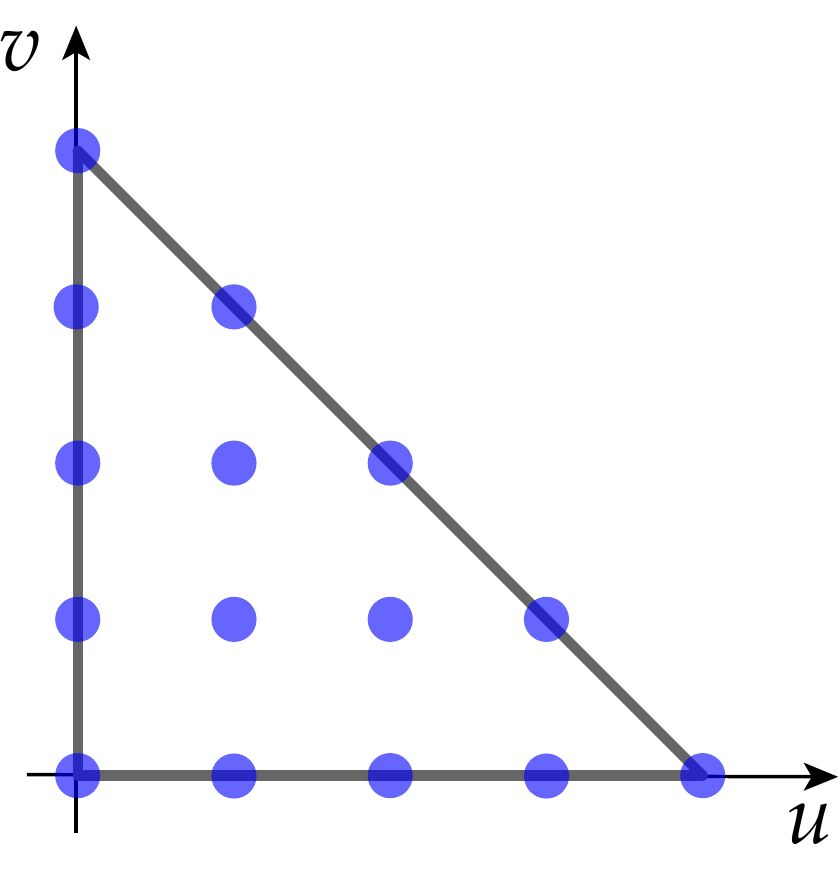}
    \caption{Fourth-order nodes.}
    \label{fig_tri4}
  \end{subfigure}
  \caption{Various discretization nodes in the $uv$-plane describing
    curvilinear triangles}
  \label{fig_gmshnodes}
\end{figure}

\subsection{Function discretization on triangles}

Not only does the surface $\Gamma$ need to be discretized, but so do
any layer potential densities defined along it. Assuming each surface
triangle is given as a map of the standard triangle $T_0$, we
discretize functions along $\Gamma$ by sampling values in the $uv$-plane
of each triangle $T_j$, i.e. $f_j = f_j(u,v)$. Sampling functions on
$T_0$ via points in the $uv$-plane
known as Vioreanu-Rokhlin nodes yields very stable
interpolation formulae, as well as nearly-Gaussian quadrature schemes
for bivariate polynomials~\cite{vioreanu_2014}. 
Subsequent evaluation of functions at points other than the
Vioreanu-Rokhlin nodes can be performed by first computing the
coefficients of the corresponding Koornwinder expansion, and then
evaluating this expansion at an arbitrary point $u,v$. Koornwinder
polynomials can be computed using standard recurrence relations for
Jacobi and Legendre polynomials~\cite{nist}.

\subsection{Computation of layer potentials}
\label{sec_laypot}

The integral equation
solvers used in the subsequent numerical examples section
directly construct the matrices which apply layer potentials
one at a time using either generalized Gaussian
quadrature or adaptive integration, depending on whether or not the
interaction is singular or not.
The individual discretized layer potentials are combined to form the
overall system matrix~$\mtx{A}$, which (unless otherwise noted)
is a discretization of the integral
operator
\begin{equation}
  \cA = -\frac{I}{4} - 2\cS H \cS' - S\left( \cS'' + \cD' \right)
  + \cD^2 + \cS \cW \cS.
\end{equation}
Our solver is based on the one
described in~\cite{bremer_2012c}.
We discretize $\cS$, $\cD$, $\cS'$, and $\cS'' + \cD'$ in this manner.
While some numerical loss of precision in evaluating the kernel of
$\cS'' + \cD'$ occurs, it does not dominate the accuracy of
the overall calculation.
This procedure scales as $\mathcal O(N^2)$, where $N$ denotes the number of
discretization nodes of the unknown along $\Gamma$. Depending on $N$,
the resulting linear system is solved via LAPACK's $\mtx{LU}$
factorization~\cite{lapack} at a cost of $\mathcal O(N^3)$, or via
GMRES at a cost of $\mathcal O(N^2)$.

To this end, we describe the discretization procedure in the case of
the single-layer potential due to a density $\sigma$ located along the
surface $\Gamma$. The discretization procedure for other layer
potentials is nearly identical, with only the kernel changing.
Since $\Gamma$ is described by a sequence of $\ntri$
curvilinear triangles $T_j$, we can rewrite this integral as:
\begin{equation}\label{eq_ssigma}
  \begin{aligned}
    \cS\sigma(\bx) &= \int_\Gamma G(\bx-\bx') \, \sigma(\bx') \, da(\bx') \\
    &= \sum_{j = 1}^{\ntri} \int_{T_j} G(\bx-\bx') \, \sigma(\bx') \, da(\bx') \\
    &= \sum_{j = 1}^{\ntri}  \int_0^1 \int_0^{1-v} 
    \frac{\sigma(\bx^j(u,v))}{4\pi |\bx - \bx^j(u,v)|} \, 
    \sqrt{|g(u,v)|} \, du \, dv,
 \end{aligned}
\end{equation}
where, as in~\eqref{eq_xji}, $\bx^j$ is the map from the standard triangle
$T_0$ to $T_j$ along $\Gamma$.

For a particular triangle $T_j$, if $\bx \in T_j$, then 
specialized quadratures for weakly-singular kernels
must be used to evaluate the integral
\begin{equation}
\cS_{T_j}\sigma(\bx) =   \int_{T_j} G(\bx-\bx') \, \sigma(\bx') \, da(\bx').
\end{equation}
For targets $\bx \notin T_j$, the integrand in
the above expression is smooth, and can be evaluated using adaptive
quadrature in parameter space, $T_0$.

In practice, if two surface triangles are far enough apart, the smooth
Vioreanu-Rokhlin quadrature can be used instead of adaptive
integration. This discretization scheme is more amenable to
acceleration via fast multipole methods, but we do not discuss the
procedure here as our solver implements the above scheme.

\subsubsection{Singular interactions}

To compute the weakly-singular integrals present in the evaluation of
layer potentials, for example, computing $\cS_T \sigma(\bx_i)$ when
$\bx_i$ is a discretization node located on triangle $T$, we use
generalized Gaussian quadratures designed by Bremer and
Gimbutas~\cite{bremer_2012c}. In short, after a precomputation of
quadrature nodes and weights $u_{ij},v_{ij},w_{ij}$ on the standard
triangle~$T_0$ (which depend on the particular target location
$u_i, v_i$), the self-interaction integrals are computed as:
\begin{equation}
  \begin{aligned}
  \cS_T \sigma(\bx_i) &= \iint_T G(\bx_i,\bx') \, \sigma(\bx') \,
  da(\bx') \\
  &= \iint_{T_0} G(\bx_i, \bx'(u,v)) \, \sigma(u,v) \,
  \sqrt{|g(u,v)|} \, du \, dv \\
  &\approx \sum_{j = 1}^{\nquad} w_{ij} \, G(\bx_i,\bx'(u_j, v_j)) \,
  \sigma(u_j, v_j) \,   \sqrt{|g(u_j,v_j)|}.
  \end{aligned}
\end{equation}
In general, the quadrature nodes do \emph{not} coincide with the nodes
at which $\sigma$ has been sampled. Therefore, $\sigma$ must be
interpolated to each of the quadrature nodes. We then have that
\begin{equation}
  \cS_T \sigma(\bx_i) = \sum_{j = 1}^{\nquad} w_{ij} \,
  G(\bx_i,\bx'(u_j, v_j)) \,
  \sqrt{|g(u_j,v_j)|} \sum_{k = 1}^{\npol} \entry{V}(j,k) \, \sigma_k,
\end{equation}
where $\mtx{V}$ denotes the matrix interpolating from the
Vioreanu-Rokhlin nodes $u_k,v_k$ to the quadrature nodes $u_{ij},v_{ij}$.
Using the above expression, entries in the system matrix can be
directly computed. See~\cite{bremer_2012c, bremer, bremer_2013} for more
information regarding the construction of similar quadratures.

\subsubsection{Nearly-singular interactions}

For integrals with \emph{nearly-singular} kernels, i.e. those
integrals~\eqref{eq_ssigma}
with the point $\bx$ near to the triangle $T$ (usually located on an
adjacent triangle), adaptive integration is
performed in order to accurately construct the system matrix.
In particular, given the Vioreanu-Rokhlin interpolation nodes 
on triangle $T$ (equivalently in parameter-space, on $T_0$),
the density $\sigma$ can be expressed as:
\begin{equation}\label{eq_kexp}
\sigma(u,v) = \sum_{m+n\leq p} c_{mn} \, K_{mn}(u,v).
\end{equation}
The coefficients $c_{mn}$ can be computed using
the point values of $\sigma$ at the
interpolation nodes $u_j,v_j$ and the values of $K_{mn}$ at these same
nodes. Restricting the domain of integration
to a single curvilinear triangle $T$, inserting this expansion into
integral~\eqref{eq_ssigma}, we have
\begin{equation}\label{eq_adap}
  \begin{aligned}
    \cS_T \sigma(\bx) &= \iint_T G(\bx,\bx') \, 
    \sigma(\bx') \, da(\bx') \\
    &= \iint_{T_0} G(\bx,\bx'(u,v)) \, 
     \sum_{m,n} c_{mn} \, 
    K_{mn}(u,v) \,  
    \sqrt{|g(u,v)|} \, du \, dv \\
    &= \sum_{m,n} c_{mn} \iint_{T_0} 
    G(\bx,\bx'(u,v)) \, K_{mn}(u,v) \, 
    \sqrt{|g(u,v)|} \, du \, dv.
 \end{aligned}
\end{equation}
If the matrix mapping point values of $\sigma$ to coefficients
in~\eqref{eq_kexp} is denoted by $\mtx{U}$, then each $c_\ell$ is
given by
\begin{equation}
c_\ell = \sum_{k=1}^{\npol} \entry{U}(\ell,k) \, \sigma_k,
\end{equation}
where $\npol$ denotes the number of Koornwinder polynomials of
degree~$\leq p$, and the $c_{mn}$'s have been linearly ordered.
In fact, $\npol = (p+1)(p+2)/2$.
Inserting this expression into~\eqref{eq_adap} yields
\begin{equation}
  \begin{aligned}
    \cS_T \sigma(\bx) &= \sum_{\ell=1}^{\npol} 
    \sum_{k=1}^{\npol} \entry{U}(\ell,k) \, \sigma_k \iint_{T_0} 
    G(\bx,\bx'(u,v)) \, K_\ell(u,v)  \, 
    \sqrt{|g(u,v)|} \, du \, dv \\
    &= \sum_{k=1}^{\npol} \sigma_k 
    \sum_{\ell=1}^{\npol} \entry{U}(\ell,k) \, \cS_T K_\ell(\bx).
 \end{aligned}
\end{equation}
For each target $\bx_i$ near to triangle $T$, the values 
$\cS_T K_\ell(\bx_i)$ can be computed via adaptive integration
on $T_0$ and
summed across column elements of $\mtx{U}$ to compute the contribution
of the point value $\sigma_k$ to the evaluation of $\cS_T\sigma
(\bx_i)$.
Entries in the discretized system matrix $\mtx{S}$ of $\cS$ are
then given by
\begin{equation}
\entry{S}(i,j) = \sum_{\ell = 1}^{\npol} \entry{U}(\ell,j) \, \cS_T
K_\ell(\bx_i).
\end{equation}
This procedure is then repeated for all targets not contained on~$T$.

\section{Numerical examples}
\label{sec_numerical}

In this section we provide several numerical experiments demonstrating
the integral equation methods for solving
the Laplace-Beltrami equation. Each of the following numerical
examples was implemented in Fortran 90, compiled with the Intel
Fortran compiler (using the MKL libraries), and executed
on a 60-core machine with 4 Intel Xeon processors
with 1.5Tb of shared RAM. When possible, OpenMP parallelization was
used (dense matrix-vector multiplies, matrix generation, etc.).

\subsection{Convergence on the sphere}
\label{sec_sphere}

In this first numerical experiment,
we solve $\surflap \phi = f$ on the unit sphere.
The numerical  result is compared to the exact calculation when~$f$ is
a spherical harmonic, whose solution is known
analytically~\cite{vico_2014}.

We define the spherical harmonic of degree~$\ell$ and order~$m$ to be
$Y^m_\ell$, $|m| \leq \ell$, given by
\begin{equation}
Y^m_\ell(\theta,\phi) = P^m_\ell(\cos\theta) \, e^{im\phi}.
\end{equation}
We have implicitly normalized the associated Legendre function
$P^m_\ell$ so that
\begin{equation}
\int_0^{2\pi} \int_0^\pi \left| Y^m_\ell(\theta,\phi) \right|^2 \, \sin\theta \,
d\theta \, d\phi = 1.
\end{equation}
The functions $Y^m_\ell$ are orthonormal, and are
the eigenfunctions of the Laplace-Beltrami
operator on the sphere with eigenvalues $\lambda^m_\ell =
-\ell(\ell+1)$.
Using this fact, we can construct an explicit, exact solution to the
Laplace-Beltrami problem
\begin{equation}\label{eq_surfylm}
\surflap \phi = Y^m_\ell.
\end{equation}
The solution $\phi$ is given analytically as
$\phi = -Y^m_\ell/\ell(\ell+1)$.
We verify our numerical solver by reformulating
problem~\eqref{eq_surfylm} in integral form as before:
\begin{equation}
\cS \left( \surflap + \cW \right) \cS \sigma = \cS Y^m_\ell.
\end{equation}
The right-hand side, $\cS Y^m_\ell$ is computed numerically, and the
integral equation for $\sigma$
is discretized and solved using Gaussian elimination.
The solution $\phi$ is then computed numerically as $\phi = \cS
\sigma$ and compared with the exact solution.

\begin{figure}[!t]
  \begin{center}
    \begin{subfigure}[b]{.3\linewidth}
      \centering
      \includegraphics[width=.95\linewidth]{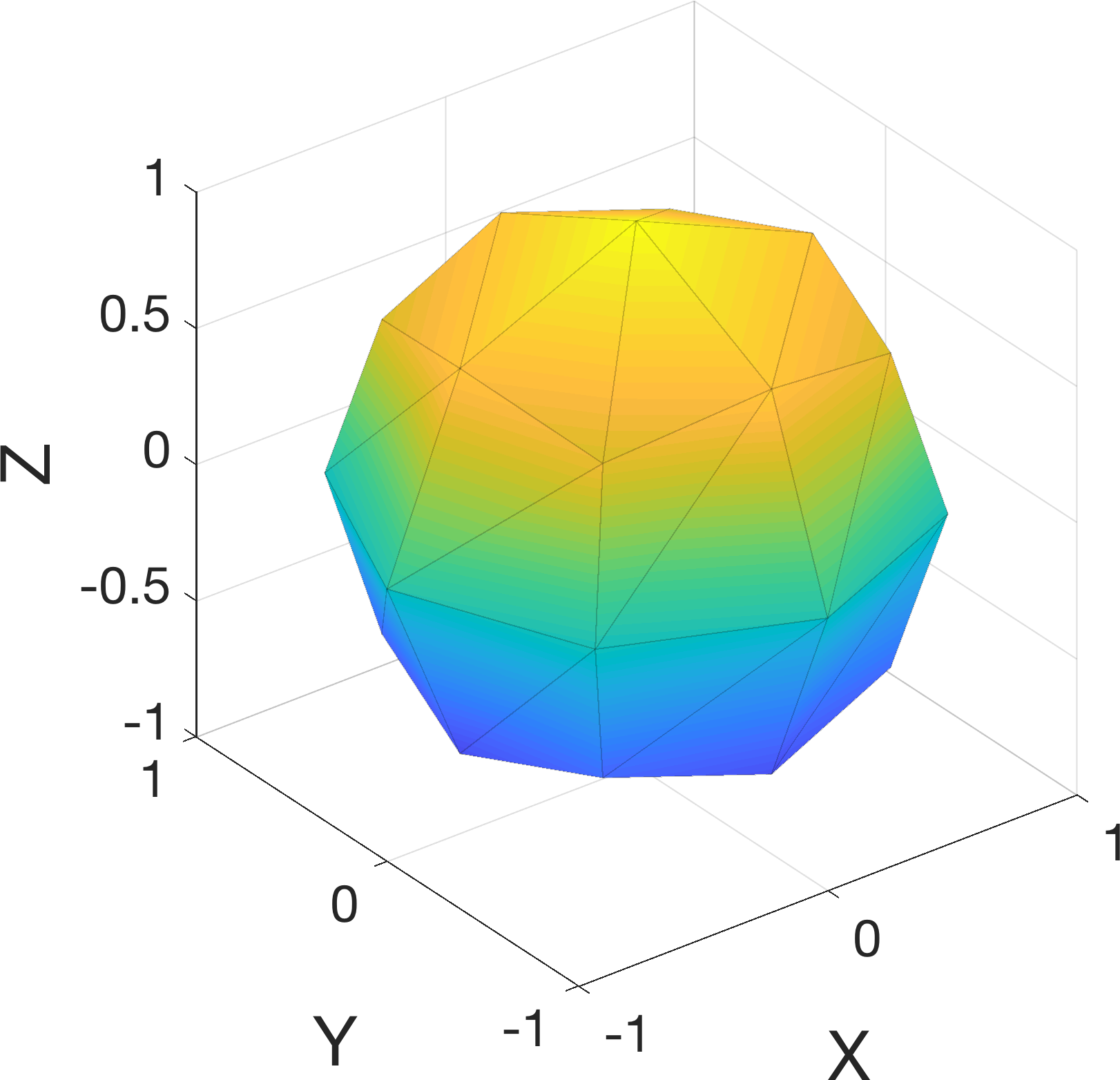}
      \caption{$\ntri = 48$.}
    \end{subfigure}
    \quad
    \begin{subfigure}[b]{.3\linewidth}
      \centering
      \includegraphics[width=.95\linewidth]{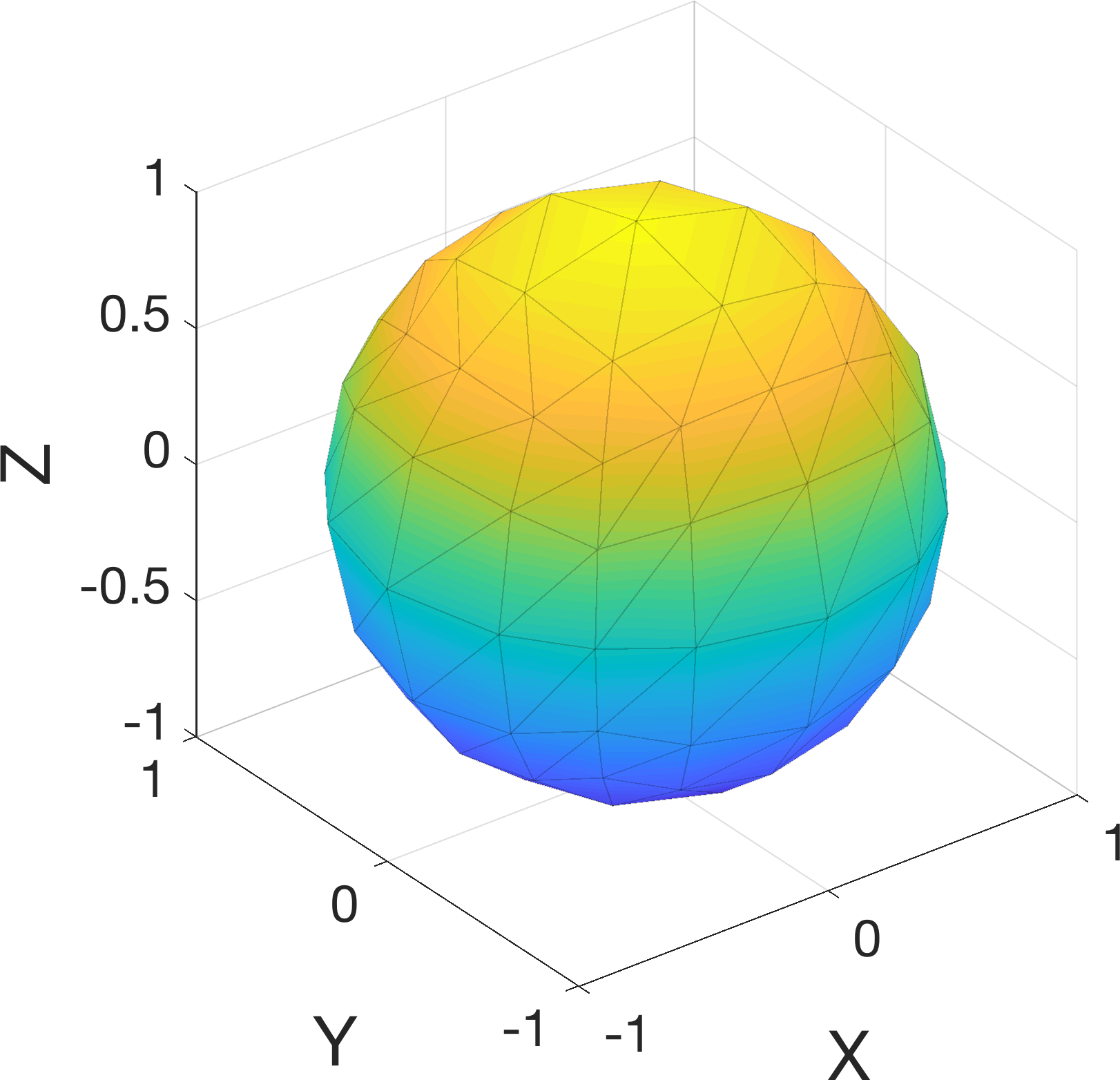}
      \caption{$\ntri = 192$.}
    \end{subfigure}
    \quad
    \begin{subfigure}[b]{.3\linewidth}
      \centering
      \includegraphics[width=.95\linewidth]{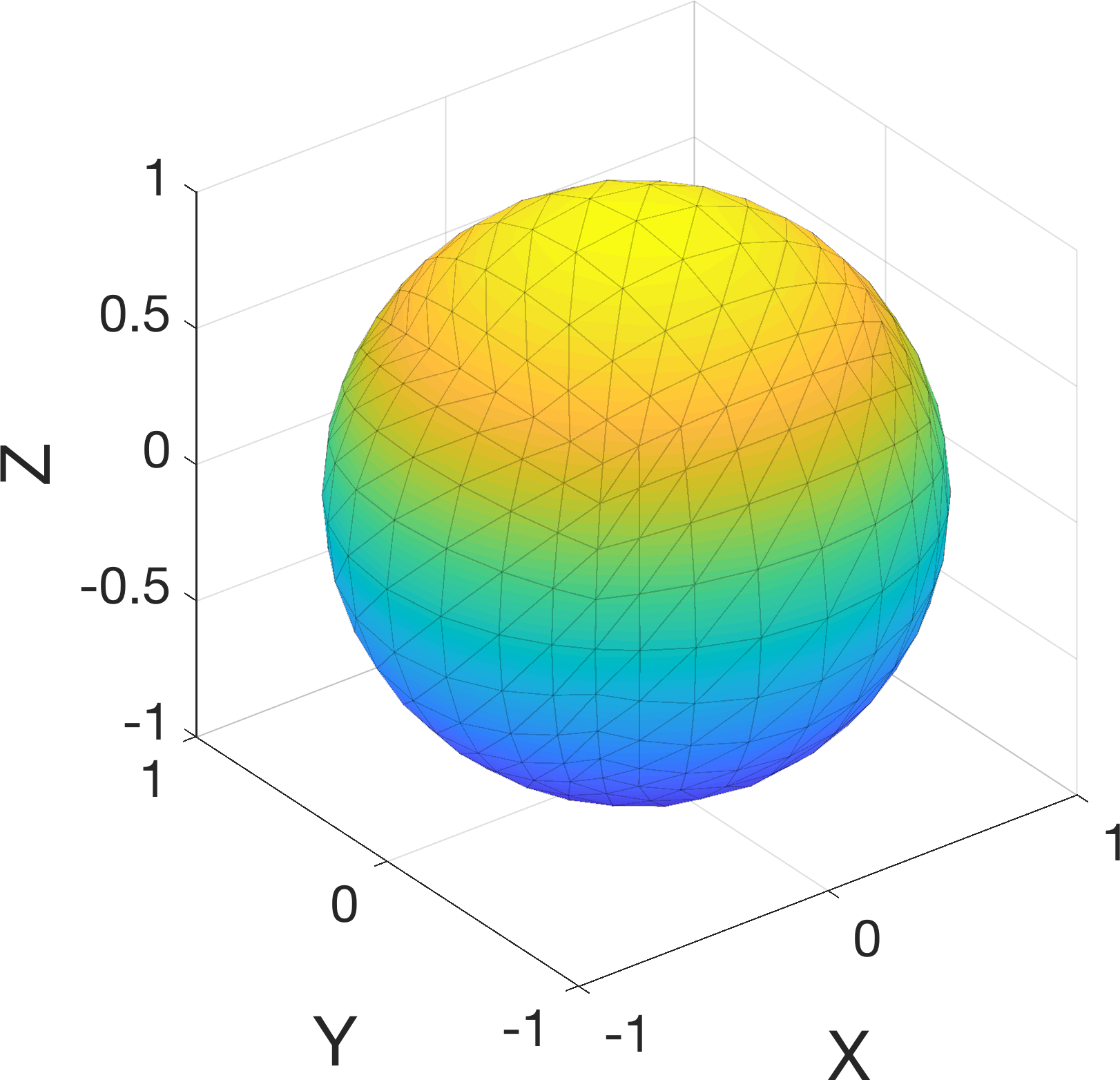}
      \caption{$\ntri = 768$.}
    \end{subfigure}
    \caption{A spherical triangulation computed by projecting an
      isotropic triangulation of the surface of a cube. Note that only
      flat projections of curvilinear triangles are plotted,
      analytic formulae are
    used to evaluate at interior points of each triangle.}
    \label{fig_proj}
  \end{center}
\end{figure}

% Furthermore, it can be shown that $\cS$ is a
% diagonal operator in the spherical harmonic basis:
% \begin{equation}
%   \cS Y^m_\ell = \frac{1}{2\ell + 1} Y^m_\ell.
% \end{equation}
% Combining these relationships, we see that:
% \begin{equation}
% \cS \surflap \cS Y^m_\ell = 
%  \frac{-\ell(\ell+1)}{(2\ell+1)^2} Y^m_\ell.
% \end{equation}

% Principal parts of the other layer potential operators can be
% calculated as:
% \begin{equation}
%   \begin{aligned}
%     \cD Y^m_\ell &= - \frac{1}{2(2\ell + 1)} Y^m_\ell,
%     &\qquad \cS' Y^m_\ell &= - \frac{1}{2(2\ell + 1)} Y^m_\ell,\\
%     \cD' Y^m_\ell &= - \frac{\ell(\ell+1)}{2\ell + 1} Y^m_\ell,
%     & \cS'' Y^m_\ell &= - \frac{\ell(\ell+1)}{2\ell + 1} Y^m_\ell.
% \end{aligned}
% \end{equation}

A triangulation of the sphere was computed by projecting an isotropic
triangulation of the surface of a cube onto the surface of the sphere;
see Figure~\ref{fig_proj} for a depiction of this.
Interpolation along the surface and subsequent derivatives can therefore
be computed exactly in order to study the behavior of the
integral equations, and not be limited by the order of discretization
of the surface.
Non-self interactions were computed using adaptive quadrature
with an absolute tolerance of at least $10^{-10}$.
Convergence results in $L^2$ along the sphere are given in
Table~\ref{tab_ylm}. The column labeled $p$ denotes the order of
discretization of $\sigma$ on each triangle, and $\npts$ denotes the
total number of discretization points: $\npts = \ntri (p+1)(p+2)/2$.
A numerical approximation to
the continuous $L^2$ norm of the error is given by
\begin{equation}
  \begin{aligned}
||\phi - \phiexact||_\Gamma^2 &= \int_\Gamma |\phi - \phiexact|^2 \,
da \\
&\approx \sum_{j=1}^{\npts} w_j \, \left| \phi_j +
  \frac{Y_\ell^m (\bs_j) }{\ell(\ell+1)}  \right|^2,
\end{aligned}
\end{equation}
where~$\bs_j$ denotes the $j$th discretization node on the sphere
and $w_j$ the corresponding $j$th smooth quadrature weight at location
$\bs_j$.
The mean of the solution $\phi$ was calculated similarly as
\begin{equation}
  \begin{aligned}
    \text{mean of } \phi &= \int_\Gamma \phi \, da \\
    &\approx \sum_{j=1}^{\npts} w_j \, \phi_j.
  \end{aligned}
\end{equation}

Examining these results, we see that high-order convergence is
obtained in $\phi$ on the sphere, appropriate with the order of
discretization. Since the geometry information is available exactly,
with $p=4$ the overall convergence order should be 4th order and with
$p=8$ the overall convergence order should be 8th order.
This high-order convergence is only made possible by
the high-digit accuracy of the generalized Gaussian quadrature used
for computing the weakly-singular integrals and the use of an exact,
analytically parameterized boundary. Furthermore, the condition number
of the discretization system in most cases was $\leq 10$, therefore
not many digits were lost to numerical round-off when performing
Gaussian elimination.

\begin{table}[!t]
  \begin{center}
    \caption{Convergence on the sphere of the Laplace-Beltrami
      integral equation.}
    \label{tab_ylm}
    \begin{subtable}[b]{.45\linewidth}
      \begin{center}
      \caption{Convergence for $\surflap \phi = Y^1_1$.}
      \begin{tabular}{|ccc|cc|} \hline
        $p$ & $\ntri$ & $\npts$ & $L^2$ error
              & mean of $\phi$ \\ \hline
        4 & 48 & 720 &  $9.0 \cdot 10^{-6}$ & $3.3 \cdot 10^{-16}$ \\ \hline
        4 & 192 & 2880&  $1.4 \cdot 10^{-7}$ & $5.9 \cdot 10^{-16}$ \\ \hline
        4 & 768 & 11520 & $1.6 \cdot 10^{-9}$ & $1.2 \cdot 16^{-15}$ \\ \hline
      \end{tabular}
      \end{center}
    \end{subtable} \hfill
    \begin{subtable}[b]{.45\linewidth}
      \begin{center}
      \caption{Convergence for $\surflap \phi = Y^1_1$.}
      \begin{tabular}{|ccc|cc|} \hline
        $p$ & $\ntri$ & $\npts$ & $L^2$ error
              & mean of $\phi$ \\ \hline
        8 & 48 & 2160 & $2.4 \cdot 10^{-8}$  & $-6.4 \cdot 10^{-16}$ \\ \hline
        8 & 192 & 8640 & $4.7 \cdot 10^{-11}$ & $-2.3 \cdot 10^{-17}$ \\ \hline
        8 & 768 & 34560 & $5.8 \cdot 10^{-14}$ & $-9.6 \cdot 10^{-16}$ \\ \hline
      \end{tabular}
      \end{center}
    \end{subtable}\\
    \vspace{\baselineskip}
    \begin{subtable}[b]{.45\linewidth}
      \begin{center}
      \caption{Convergence for $\surflap \phi = Y^6_7$.}
      \begin{tabular}{|ccc|cc|} \hline
        $p$ & $\ntri$ & $\npts$ & $L^2$ error
              & mean of $\phi$ \\ \hline
        4 & 48 & 720 &  $3.1 \cdot 10^{-4}$ & $-6.2 \cdot 10^{-8}$ \\ \hline
        4 & 192 & 2880 & $8.4 \cdot 10^{-6}$ & $-1.5 \cdot 10^{-9}$ \\ \hline
        4 & 768 & 11520 &  $1.2 \cdot 10^{-7}$ & $-4.7 \cdot 10^{-12}$ \\ \hline
      \end{tabular}
      \end{center}
    \end{subtable} \hfill
    \begin{subtable}[b]{.45\linewidth}
      \begin{center}
      \caption{Convergence for $\surflap \phi = Y^6_7$.}
      \begin{tabular}{|ccc|cc|} \hline
        $p$ & $\ntri$ & $\npts$ & $L^2$ error
              & mean of $\phi$ \\ \hline
        8 & 48 & 2160 &  $2.3 \cdot 10^{-6}$ & $7.2 \cdot 10^{-10}$ \\ \hline
        8 & 192 & 8640 & $8.1 \cdot 10^{-9}$ & $2.3 \cdot 10^{-14}$ \\ \hline
        8 & 768  & 34560 & $1.0 \cdot 10^{-11}$ & $-2.0 \cdot 10^{-16}$ \\ \hline
      \end{tabular}
      \end{center}
    \end{subtable}
  \end{center}
\end{table}

\subsection{Laplace-Beltrami on general surfaces}

Except along a handful of geometries, exact solutions to the
Laplace-Beltrami equation are not known. Therefore, in order to test
the numerical integral equation solver we have constructed, a
right-hand side must be numerically generated for which we know the
solution \emph{a priori}. This can be accomplished by using the
relationship between the surface Laplacian and the volume Laplacian,
given in Lemma~\ref{lem_lap}.

\begin{figure}[!t]
  \begin{center}
    \begin{subfigure}[b]{.45\linewidth}
      \centering
      \includegraphics[width=.95\linewidth]{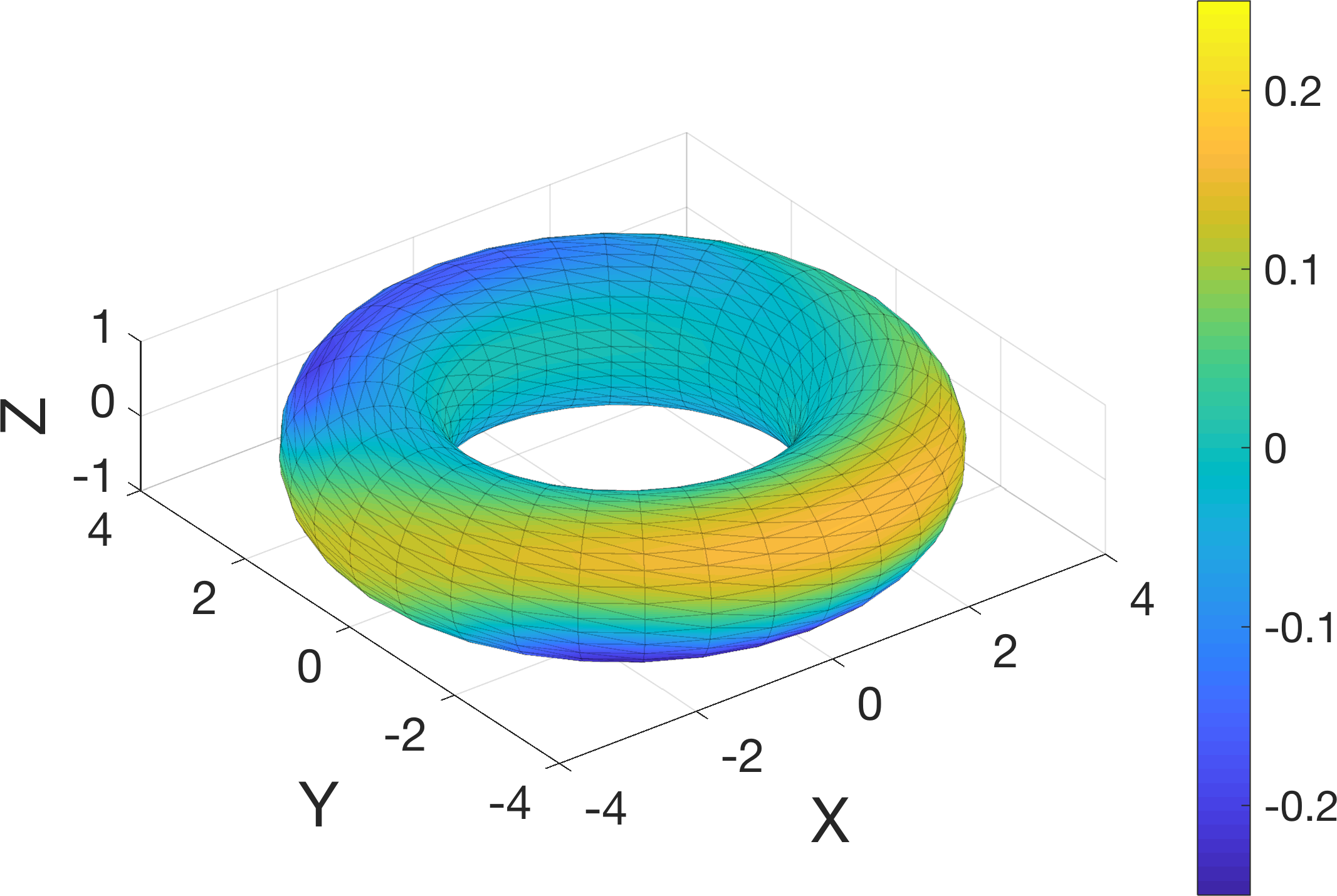}
      \caption{The right-hand side of $\surflap \phi = f$.}
      \label{fig_tor_rhs}
    \end{subfigure}
    \quad
    \begin{subfigure}[b]{.45\linewidth}
      \centering
      \includegraphics[width=.95\linewidth]{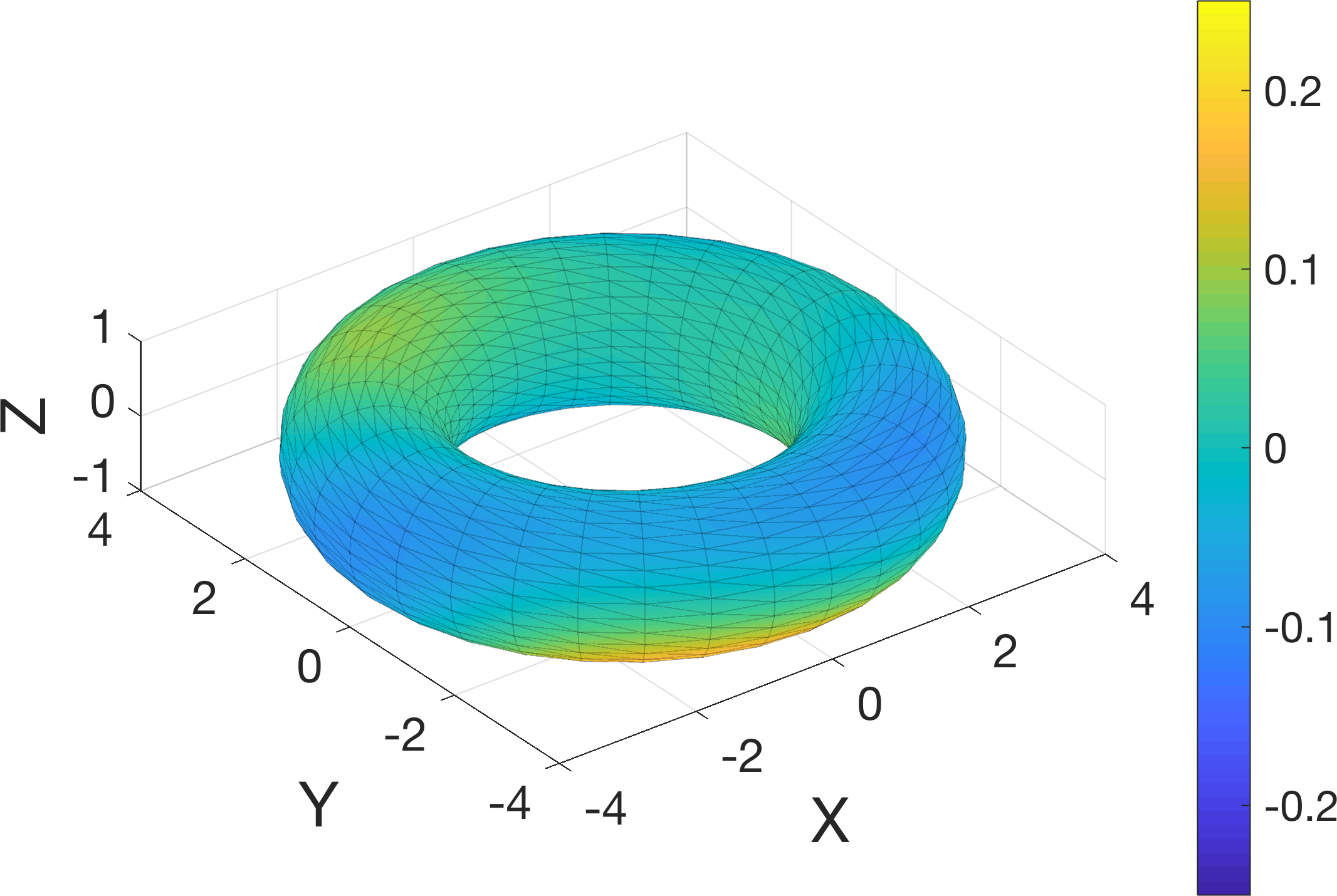}
      \caption{The solution of $\surflap \phi = f$.}
      \label{fig_tor_sol}
    \end{subfigure}
    \caption{The right-hand side and solution to $\surflap \phi = f$
      along the surface of a torus. }
    \label{fig_tor}
  \end{center}
\end{figure}

\begin{table}[b]
  \begin{center}
    \caption{Convergence for $\surflap \phi = f$ along the torus using
      an 8th-order discretization. The condition number of the
      discretized matrices for $\ntri=32$ and $\ntri=128$ was was $1.8
      \cdot 10^{1}$. The cost of computing the SVD (and therefore the
      condition number) was prohibitive for $\ntri=512$, but a similar
      estimate is expected due to the stability of second-kind
      integral equations under
      refinement~\cite{anselone_1971,atkinson_1997}.}
    \label{tab_torconv}
    \begin{tabular}{|cc|cc|} \hline
      $\ntri$ & $\npts$ & $L^2$ error
      & mean of $\phi$ \\ \hline
      32  & 1440 &  $ 1.6 \cdot 10^{-4}$ & $ 8.2 \cdot 10^{-5}$ \\ \hline
      128  & 5760 & $ 4.7 \cdot 10^{-7}$ & $ -3.3\cdot 10^{-8}$ \\ \hline
      512  & 23040 & $ 1.1 \cdot 10^{-9}$ & $ 2.1 \cdot 10^{-11}$ \\ \hline
    \end{tabular}
  \end{center}
\end{table}

In particular, let $g = g(x_1,x_2,x_3)$ be a
smooth function defined in the interior of the region bounded by
$\Gamma$ and in an exterior neighborhood of $\Gamma$. Next, define
the function $f$ along~$\Gamma$ to be
\begin{equation}\label{eq_f_rhs}
  \begin{aligned}
    f &= \surflap \, g\big|_\Gamma \\
    &= \Delta g - 2H\frac{\partial g}{\partial n} -
    \frac{\partial^2 g}{\partial n^2}.
  \end{aligned}
\end{equation}
Since $f = \surfdiv \surfgrad g$,
it is easy to show that $f$ has mean-zero by using Gauss's Theorem
along~$\Gamma$.
Then, the exact solution to the Laplace-Beltrami problem $\surflap
\phi = f$ is given by
\begin{equation}
  \phi = g\big|_\Gamma - \frac{1}{|\Gamma|} \int_\Gamma g \, da,
\end{equation}
which is to say, the projection of $g$ onto mean-zero functions
on~$\Gamma$.

We first use this approach to test the surface Laplacian solver along
a 3-to-1 torus, analytically parameterized as
\begin{equation}\label{eq_torus}
  \begin{aligned}
    x_1(u,v) &= (3 + \cos u) \cos v \\
    x_2(u,v) &= (3 + \cos u) \sin v \\
    x_3(u,v) &= \sin u,
  \end{aligned}
\end{equation}
for $(u,v) \in [0,2\pi] \times [0,2\pi]$. We define the function $g$
to be
\begin{equation}
  g(\bx) = C \sum_{j=1}^{10} \frac{1}{4\pi |\bx-\bx_j|},
\end{equation}
where the $\bx_j$ are placed randomly at a distance of 7 from the
origin and $C$ is numerically calculated so that $\int_\Gamma |f|^2 da
= 1$.
The right-hand side, $f = \surflap g\big|_\Gamma$ is shown in
Figure~\ref{fig_tor_rhs} and the solution is shown in
Figure~\ref{fig_tor_sol}. Table~\ref{tab_torconv} contains convergence
results for an 8th-order discretization of the integral equation.
The linear system was solved using GMRES
with a relative $\ell^2$-residual tolerance of $10^{-14}$.

Finally, we apply our solver to geometries constructed via Computer
Aided Design (CAD) software. The geometries in Figures~\ref{fig_gmsh1}
and~\ref{fig_gmsh2} were designed in Autodesk Fusion
360~\cite{autodesk_fusion}, exported as \texttt{.step} files, and then
meshed using 4th-order curvilinear triangles in Gmsh~\cite{gmsh}
(whose nodal locations are given in Figure~\ref{fig_tri4}). Take note
that the right-hand sides shown in Figures~\ref{fig_gmsh1_rhs}
and~\ref{fig_gmsh2_rhs} are
\emph{not} smooth. This is because of normal and curvature discontinuities
generated in either the modeling or meshing procedure. Sufficient
smoothing and optimization would have to be done in order to avoid
this phenomenon.

\begin{figure}[!t]
  \begin{center}
    \begin{subfigure}[b]{.45\linewidth}
      \centering
      \includegraphics[width=.95\linewidth]{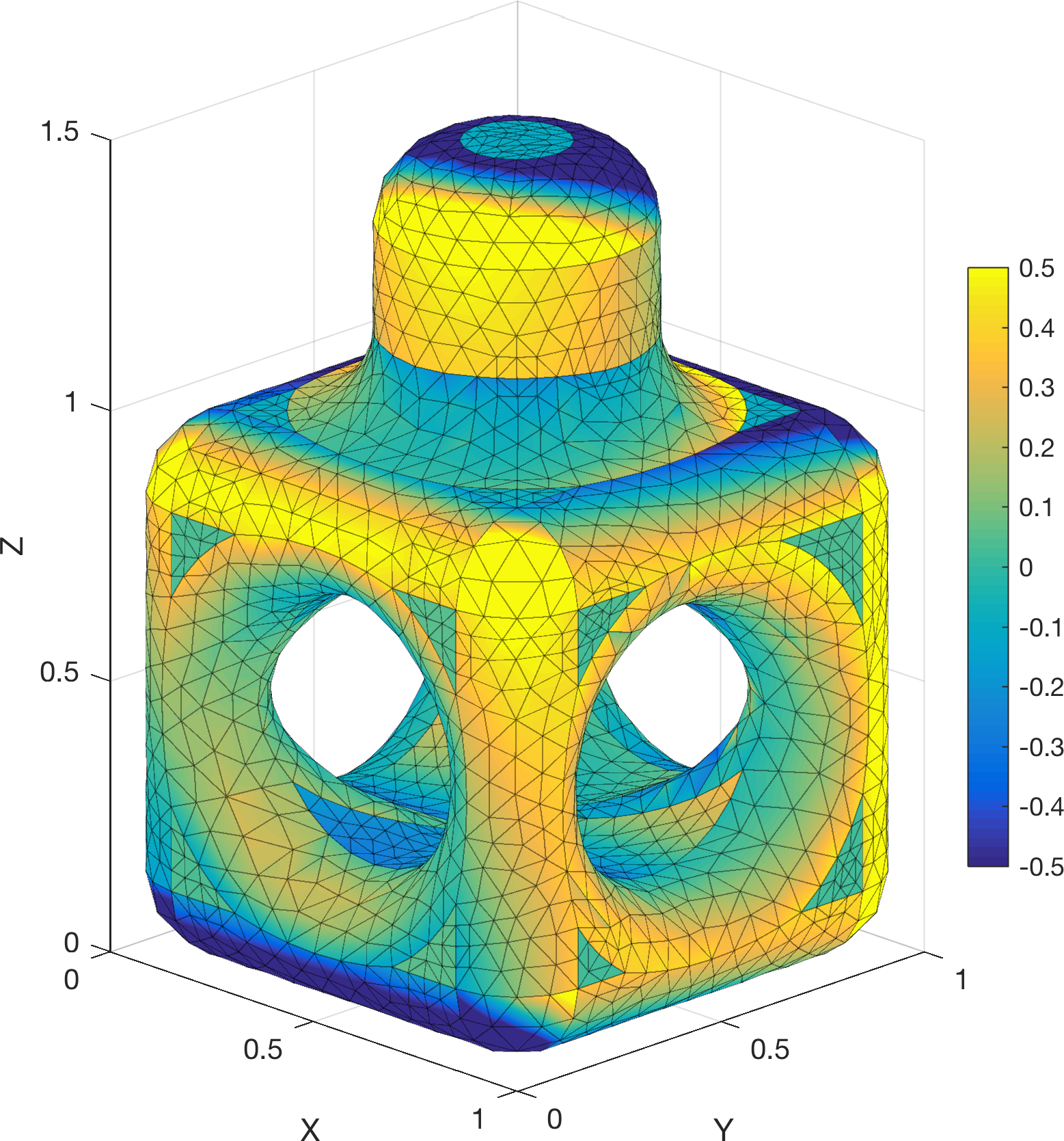}
      \caption{The right-hand side of $\surflap \phi = f$.}
      \label{fig_gmsh1_rhs}
    \end{subfigure}
    \quad
    \begin{subfigure}[b]{.45\linewidth}
      \centering
      \includegraphics[width=.95\linewidth]{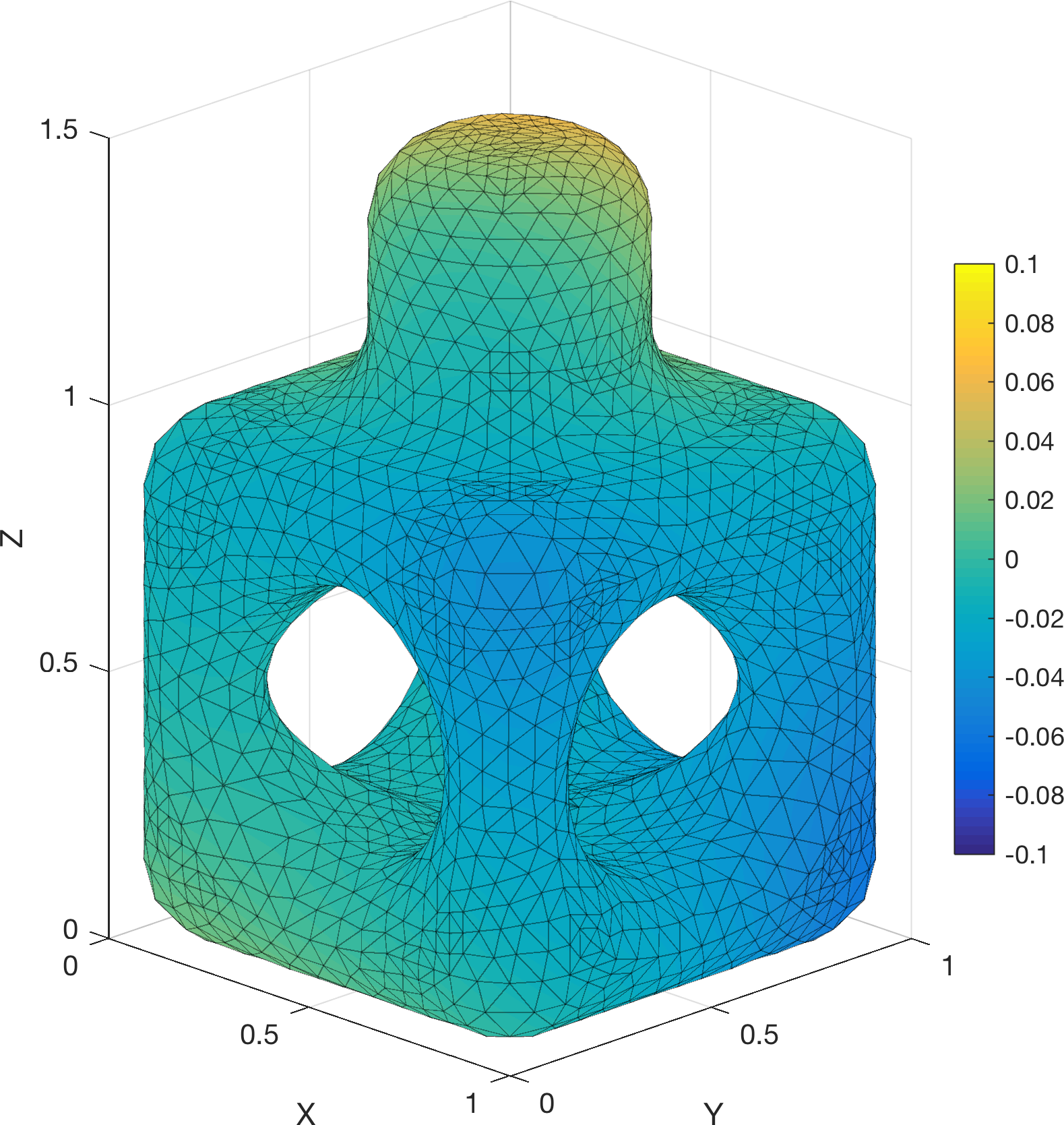}
      \caption{The solution of $\surflap \phi = f$.}
      \label{fig_gmsh1_sol}
    \end{subfigure}
    \caption{The right-hand side and solution to $\surflap \phi = f$
      along a 4th-order triangulated surface. The surface consisted of
      1412 triangles, was discretized using 8th-order interpolation
      points, resulting in 63,640 total discretization nodes. The solve
      resulted in an $L^2$ error of $2.7\cdot 10^{-2}$.
      The mean of the solution
      was $-6.8 \cdot 10^{-2}$.
      The solve required 22 GMRES iterations and achieved a
      relative residual of $8.9 \cdot 10^{-15}$.
      Note the different color scale between the images.}
    \label{fig_gmsh1}
  \end{center}
\end{figure}

\subsection{The Hodge decomposition}

As mentioned in the introduction, any tangential vector field along
smooth closed multiply connected surfaces admits what is known as a Hodge
decomposition:
\begin{equation}
\bF = \surfgrad \alpha  + \bn \times \surfgrad \beta  + \bH.
\end{equation}
Correspondingly, given a smooth vector field $\bF$, we can
decompose it into its Hodge
components by solving the following Laplace-Beltrami equations and
computing $\bH$:
\begin{equation}\label{eq_hodgesystem}
\begin{gathered}
\surflap \alpha = \surfdiv \bF, \qquad \surflap \beta = -\surfdiv (\bn
\times \bF), \\
\bH = \bF - \surfgrad \alpha - \bn \times \surfgrad \beta.
\end{gathered}
\end{equation}

To this end, we first define a smooth, arbitrary vector field $\bV$ 
in $\bbR^3$ and compute $\bF$ as its tangential projection onto 
the surface~$\Gamma$: $\bF = -\bn \times \bn \times \bV$.
The surface divergence of $\bF$ can then be calculated in terms of the
divergence of the volume vector field $\bV$ as~\cite{nedelec}:
\begin{equation}
  \surfdiv \bF = \nabla \cdot \bV - 2H \left(\bn \cdot \bV \right) -
  \frac{\partial}{\partial n } \left(\bn \cdot \bV\right).
\end{equation}
In particular, we compute $\bV = \bB$ using the Biot-Savart
Law~\cite{jackson}
 for a point current $\bL$ located at $\bx'$:
\begin{equation}
    \bB(\bx) = \frac{\bL \times (\bx - \bx')}{|\bx - \bx'|^3}.
\end{equation}
In general, the tangential projection of~$\bB$
of this type will have non-zero projections onto
all three components in the Hodge decomposition.
The right-hand sides of the Laplace-Beltrami problems
in~\eqref{eq_hodgesystem} are automatically mean-zero, as they are
exact differentials.

\begin{figure}[!t]
  \begin{center}
    \begin{subfigure}[b]{.45\linewidth}
      \centering
      \includegraphics[width=.95\linewidth]{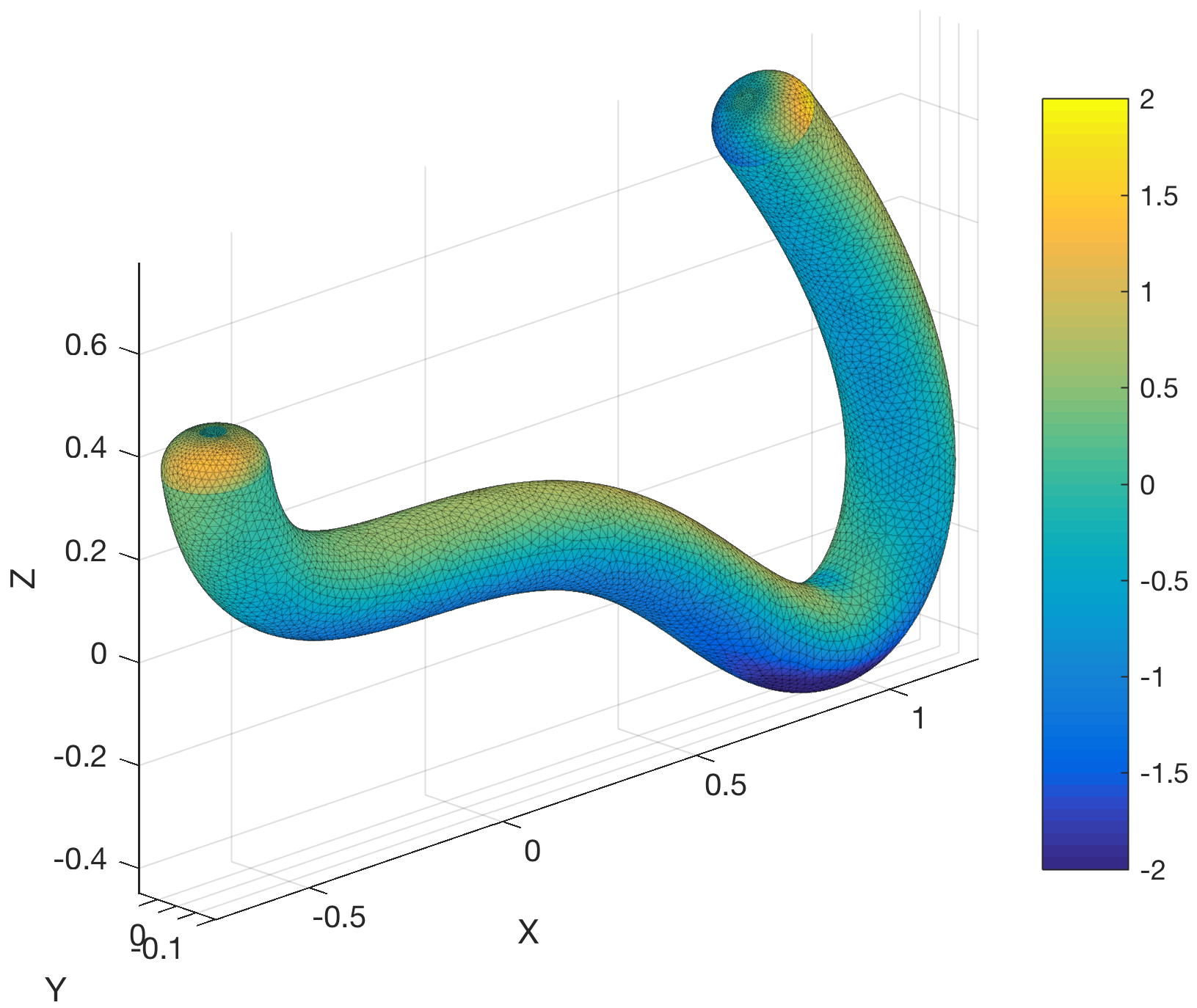}
      \caption{The right-hand side of $\surflap \phi = f$.}
      \label{fig_gmsh2_rhs}
    \end{subfigure}
    \quad
    \begin{subfigure}[b]{.45\linewidth}
      \centering
      \includegraphics[width=.95\linewidth]{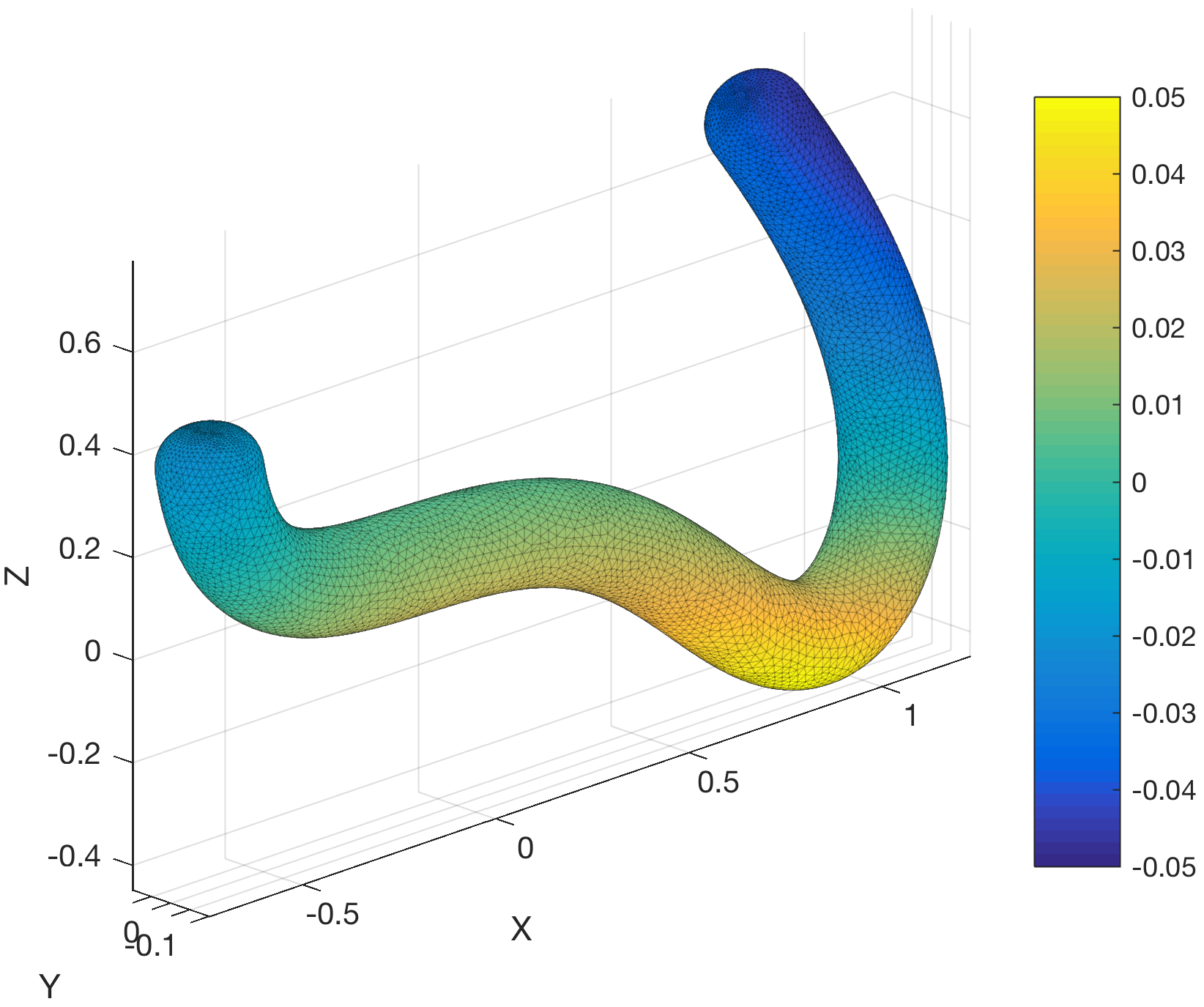}
      \caption{The solution of $\surflap \phi = f$.}
      \label{fig_gmsh2_sol}
    \end{subfigure}
    \caption{The right-hand side and solution to $\surflap \phi = f$
      along a 4th-order triangulated surface. The surface consisted of
      4604 triangles, was discretized using 4th-order interpolation
      points, resulting in 69,060 total discretization nodes. The solve
      resulted in an $L^2$ error of $1.3 \cdot 10^{-3}$.
      The mean of the solution
      was $-6.7\cdot 10^{-4}$.
      The solve required 18 GMRES iterations and achieved a
      relative residual of $1.6 \cdot 10^{-15}$.
      Note the different color scale between the images. }
    \label{fig_gmsh2}
  \end{center}
\end{figure}

Once again, we use a direct version of the solver 
describe in~\cite{bremer_2012c}. Self-interactions on triangular
patches are computed using the quadrature rules contained in the same
work, 
and non-self interactions are computed using adaptive quadrature with
a requested absolute tolerance of
$10^{-10}$. Figure~\ref{fig_torus_harm} shows the decomposition along
the torus given by
parameterization~\eqref{eq_torus}. Figure~\ref{fig_gmsh_harm} shows
the same decomposition along an arbitrary geometry constructed in
Autodesk Fusion 360, and meshed using Gmsh. Discretization information
and convergence results are
provided in the captions of each figure.
The $L^2$-norm $||\cdot ||_\Gamma$
of a vector field along $\Gamma$ is computed as:
\begin{equation}
  \begin{aligned}
    ||\bF||^2_\Gamma &=  \int_\Gamma |\bF|^2 \,  da   \\
    &\approx \sum_{j=1}^{\npts} w_j \, \bF^*_j \cdot \bF_j.
\end{aligned}
\end{equation}
As a proxy for convergence, since analytic Hodge
decompositions are not known except on trivial geometries,
we use the $L^2$ norm of the divergence and
curl of the computed harmonic component, $\bH$. Given point-wise values
of $\bH$ along $\Gamma$, denoted by $\bH_j$, it is simple to compute
$H^u_j$ and $H^v_j$, the coefficients with respect to the local basis:
\begin{equation}
  \bH_j = H^u_j \, \bx_u + H^v_j \, \bx_v.
\end{equation}
The surface divergence is then given by
\begin{equation}
  \surfdiv \bH = \frac{1}{\sqrt{|g|}} \left(
    \frac{\partial}{\partial u} \left( \sqrt{|g|} H^u \right)  +
    \frac{\partial}{\partial v}
    \left( \sqrt{|g|} H^v \right) \right) .
\end{equation}
Using the values $H^u_j$ and $H^v_j$ on each triangle,
it is possible to form the per-triangle interpolants
\begin{equation}
  \sqrt{|g|} \, H^u(u,v) = \sum_{m+n\leq p} a_{mn} \, K_{mn}(u,v), \qquad
  \sqrt{|g|} \, H^v(u,v) = \sum_{m+n\leq p} b_{mn} \, K_{mn}(u,v).
\end{equation}
Spectral differentiation can then be used to compute the partial
derivatives with respect to $u$ and $v$, and thereby enabling the
computation of $\surfdiv \bH$ (and of course, $\surfdiv \bn \times
\bH$) point by point.

\begin{figure}[!t]
  \begin{center}
    \begin{subfigure}[b]{.45\linewidth}
      \centering
      \includegraphics[width=.95\linewidth]{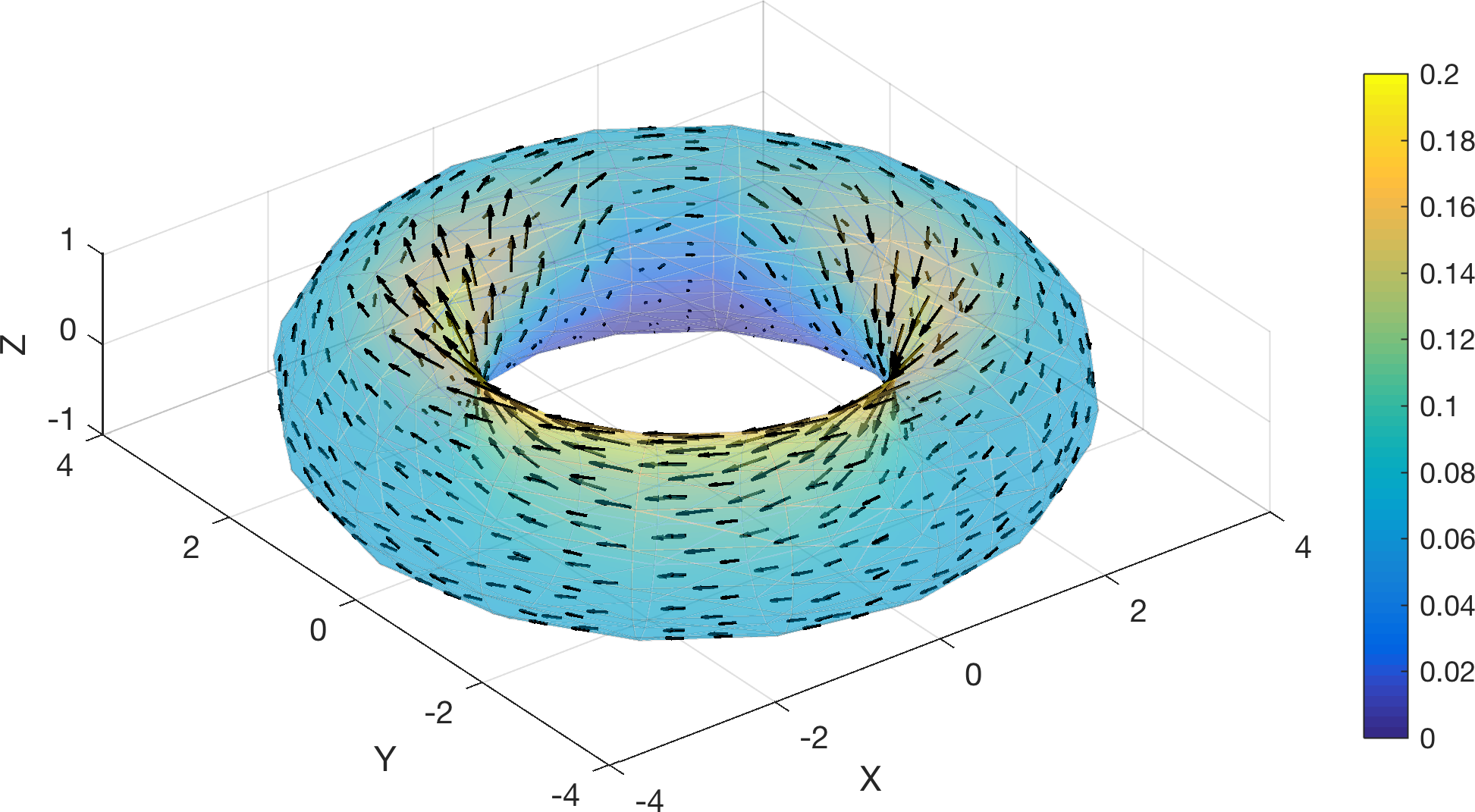}
      \caption{The tangential projection of $\bB$ onto the surface.}
      \label{fig_torus_biot}
    \end{subfigure}
    \quad
    \begin{subfigure}[b]{.45\linewidth}
      \centering
      \includegraphics[width=.95\linewidth]{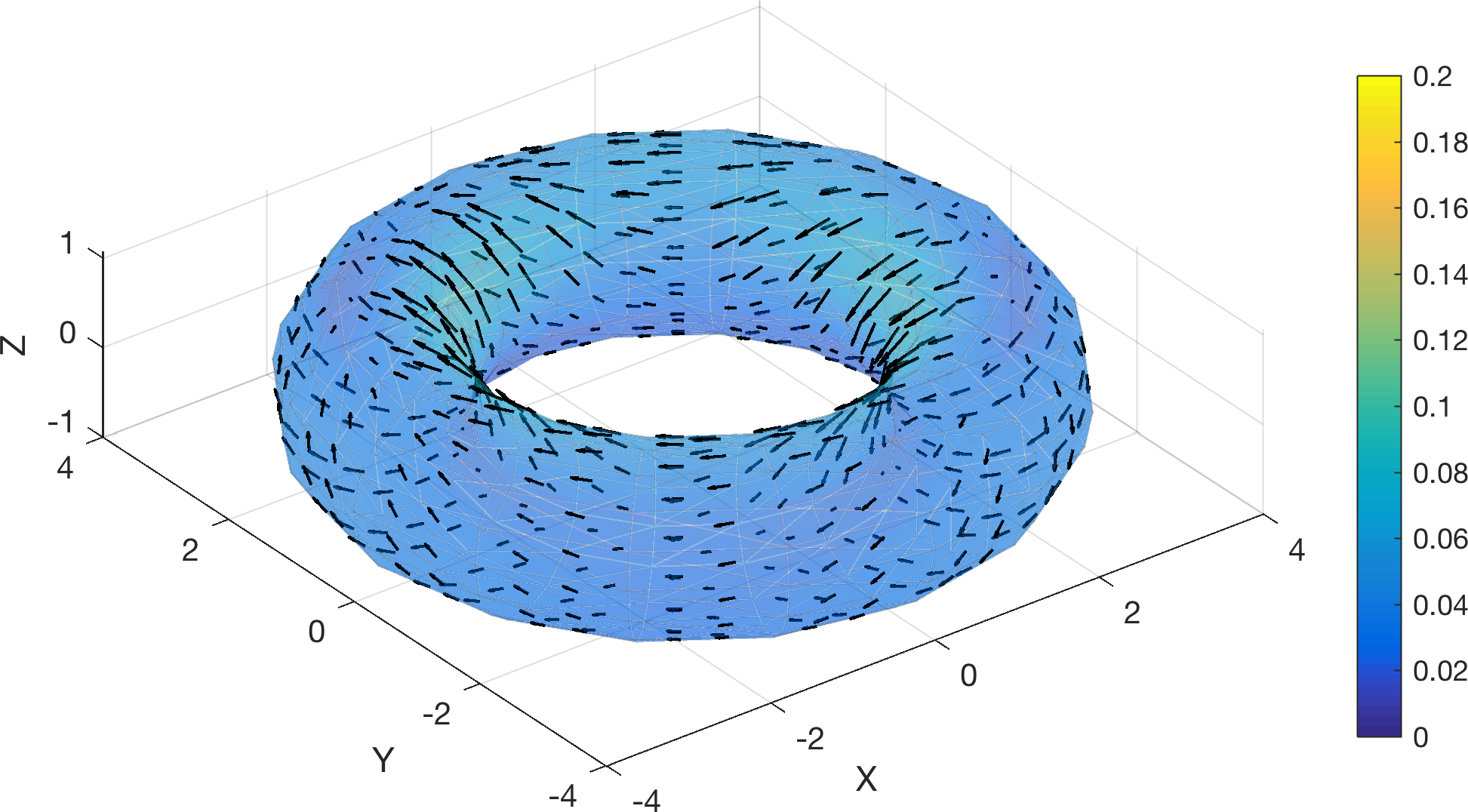}
      \caption{The curl-free component, $\surfgrad \alpha$.}
      \label{fig_torusa}
    \end{subfigure}
    \begin{subfigure}[b]{.45\linewidth}
      \centering
      \includegraphics[width=.95\linewidth]{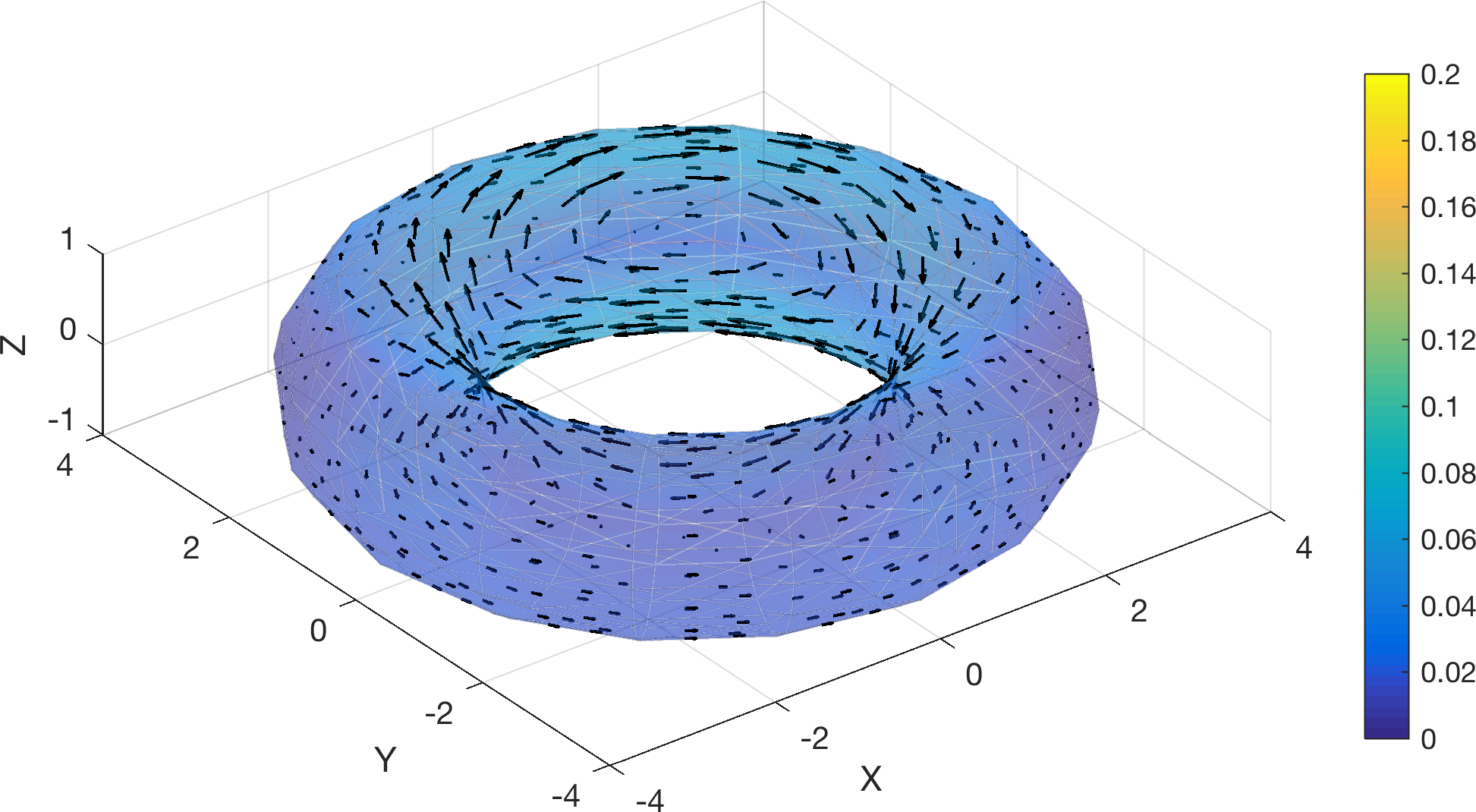}
      \caption{The divergence-free component, $\bn \times \surfgrad
        \beta$.}
      \label{fig_torusb}
    \end{subfigure}
    \quad
    \begin{subfigure}[b]{.45\linewidth}
      \centering
      \includegraphics[width=.95\linewidth]{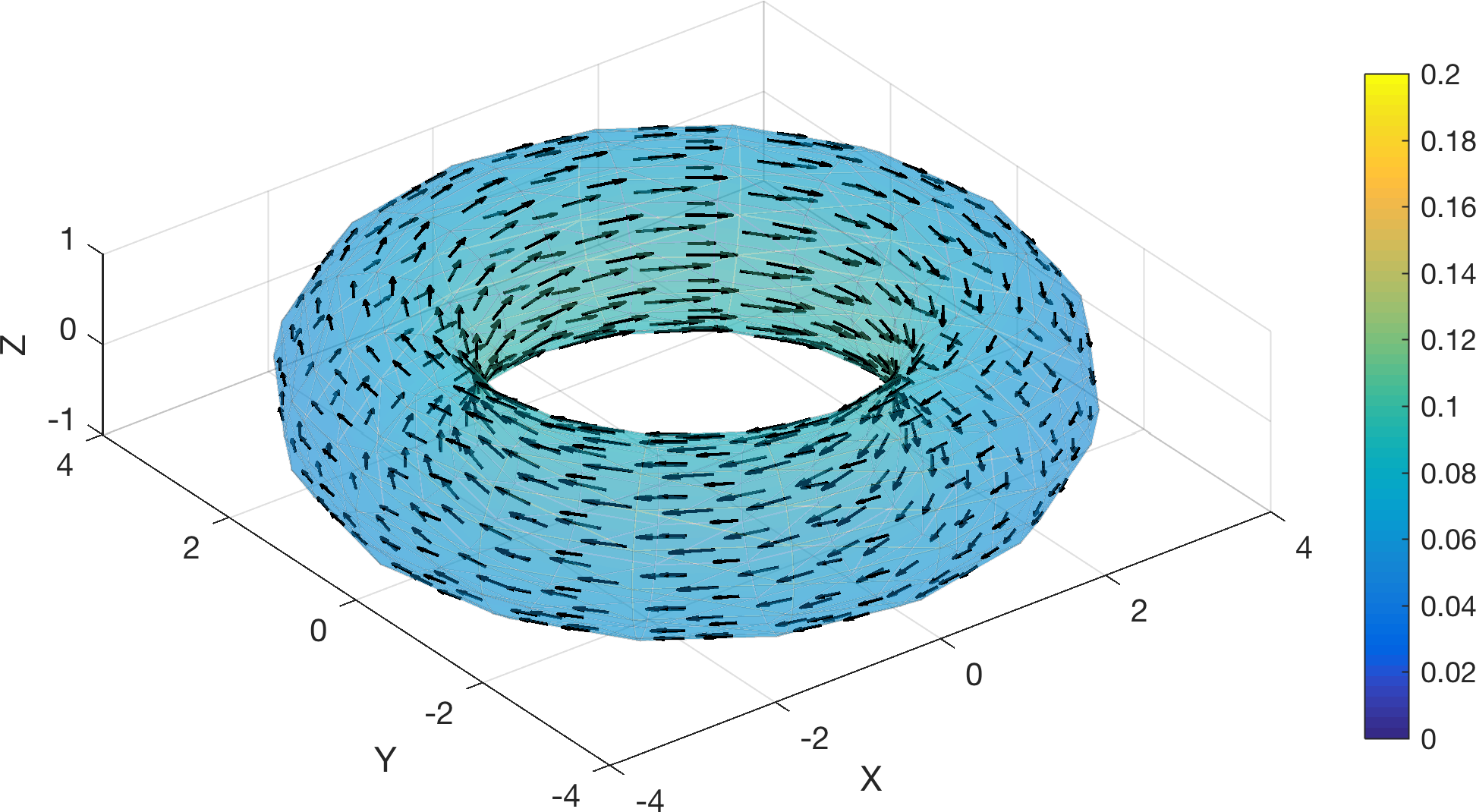}
      \caption{The harmonic component, $\bH$.}
      \label{fig_torush}
    \end{subfigure}
    \caption{The Hodge decomposition of the magnetic field computed
      using the Biot-Savart Law due to a current element located at
      (0.1, 0.2, 2.1) oriented in the direction (0.37, 0.48, -0.80). The
      strength was computed such that $||\bn \times \bB||_\Gamma = 1$.
      Shown is the direction and magnitude of the vector fields.
      The analytic parameterization of the torus was discretized using
      8th-order interpolation points on 128 triangles, yielding 5760
      discretization nodes. The linear systems were solved using GMRES
    and achieved relative residuals of less than $10^{-14}$. Using
    spectral differentiation on each triangle, $||\surfdiv
    \bH||_\Gamma = 5.9 \cdot 10^{-5}$ and $||\surfdiv \bn \times
    \bH||_\Gamma = 4.2 \cdot 10^{-5}$.}
    \label{fig_torus_harm}
  \end{center}
\end{figure}

\section{Conclusions and future directions}
\label{sec_conclusions}

In this work we have presented integral equations of the second-kind
which allow for the efficient solution to the Laplace-Beltrami problem
on arbitrary surfaces embedded in three dimensions. The integral
equations are a consequence of left- and right-preconditioners and
Calder\'on identities. A direct numerical solver was constructed to
demonstrate the usefulness of these integral equations, and high-order
accuracy was achieved for geometries which were analytically
parameterized. The numerical accuracy along lower-order triangulations
seems to be limited by inaccurate normal and curvature information,
which is a consequence of the CAD modeling or meshing
procedures. Obtaining more accurate surfaces, as well as annealing
non-smooth ones, is ongoing work.

It is worth addressing the somewhat erratic convergence behavior in
Tables~\ref{tab_ylm} and~\ref{tab_torconv}. The error in
any boundary integral equation method along surfaces comes from the
following sources: geometry approximation, singular quadrature,
nearly-singular/smooth quadrature, and the order of discretization of
the unknown. The overall convergence order is dictated by the lowest
order component of these factors. The discretization scheme upon which
this solver was based is that presented in~\cite{bremer_2012c}. The
convergence results in that work are similarly erratic, and provided
for a 12th order scheme~(see Table 2 in that work).
In theory, the scheme of this paper (and
of~\cite{bremer_2012c}) can be
made to be of arbitrarily high-order (assuming access to high-order
geometries is provided). The erratic convergence behavior
can only be attributed to the interplay between \emph{order of
  convergence} and \emph{actual accuracy obtained}. For higher-order
methods, it is often the case that the overall accuracy saturates
before the asymptotic convergence rate can be observed.

It should also be noted, that while the operator $\cS \surflap \cS$ is
equivalent to a second-kind integral operator given by
\begin{equation}
\cS \surflap \cS = -\frac{\cI}{4} +  \cD^2  - 
\cS(\cS'' + \cD') - 2\cS H \cS'  ,
\end{equation}
as shown in Lemma~\ref{lem_slaps}, it may be more computationally
efficient to iterate on the left-hand side operator composition
$\cS \surfdiv (\surfgrad \cS)$ instead of computing each of the three
compositions on the right-hand side separately.  Note that each term
on the right-hand side involves normal derivatives of layer
potentials, and each kernel is weakly singular. However, operators
such as $\cS \surfdiv$ and $\surfgrad \cS$ are of order zero and
have kernels which are
similar to Hilbert transforms along the surface. As a consequence,
a different set of singular quadratures than those used for the weakly
singular kernels are required.
 These quadratures were, in fact, computed
in~\cite{bremer_2013} but were not incorporated into the solver of
this paper. Exploring this iterative formulation is ongoing work, as well as
efficiently coupling these integral operators and curvilinear
triangulations with fast multipole methods for Laplace potentials in
three dimensions.

It is relatively straightforward to develop
similar integral formulations for other surface PDEs whose extensions
to the volume have known Green's functions, such as the
Yukawa-Beltrami equation~\cite{kropinski_2016_integral}.
For example, if $\cS_k$ denotes the single layer potential operator
for the Helmholtz equation, then
\begin{equation}
  \begin{aligned}
    \cS_k \left( \surflap + k^2 \right) \cS_k &=
    -2\cS_k  H \cS_k' - \cS_k \cS_k'' \\
    &= -2\cS_k  H \cS_k' - \cS_k \left( \cS_k'' + \cD' \right) +
    \cS_k\cD' \\
    &= -2\cS_k  H \cS_k' - \cS_k \left( \cS_k'' + \cD' \right) +
      \left( \cS_k - \cS \right) \cD' + \cS\cD' \\
      &= -\frac{1}{4}\cI + -2\cS_k  H \cS_k' -
      \cS_k \left( \cS_k'' + \cD' \right) +
      \left( \cS_k - \cS \right) \cD' + \cD^2.
  \end{aligned}
\end{equation}
The difference operators above for the mixed potentials can be shown
to be compact, resulting in a Fredholm integral equation of the second
kind.
Furthermore, combining the methods of this paper with the boundary
integral methods of~\cite{kropinski_2016_integral,kropinski_2014_fast}
would likely result in the ability to solve piece-wise
constant-coefficient related Beltrami problems on surfaces.

\begin{figure}[!t]
  \begin{center}
    \begin{subfigure}[b]{.45\linewidth}
      \centering
      \includegraphics[width=.95\linewidth]{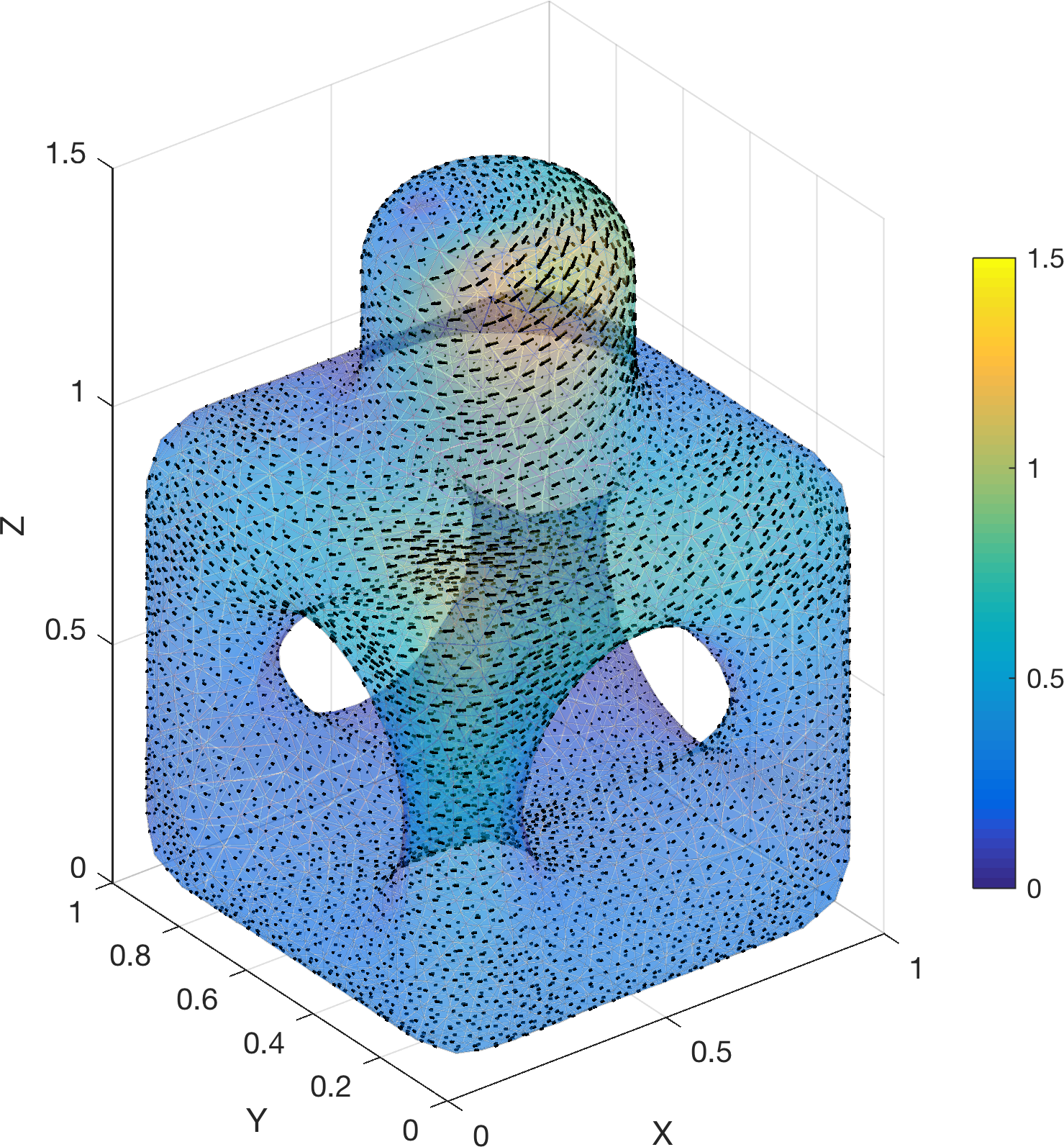}
      \caption{The tangential projection of $\bB$ onto the surface.}
      \label{fig_torus_biot}
    \end{subfigure}
    \quad
    \begin{subfigure}[b]{.45\linewidth}
      \centering
      \includegraphics[width=.95\linewidth]{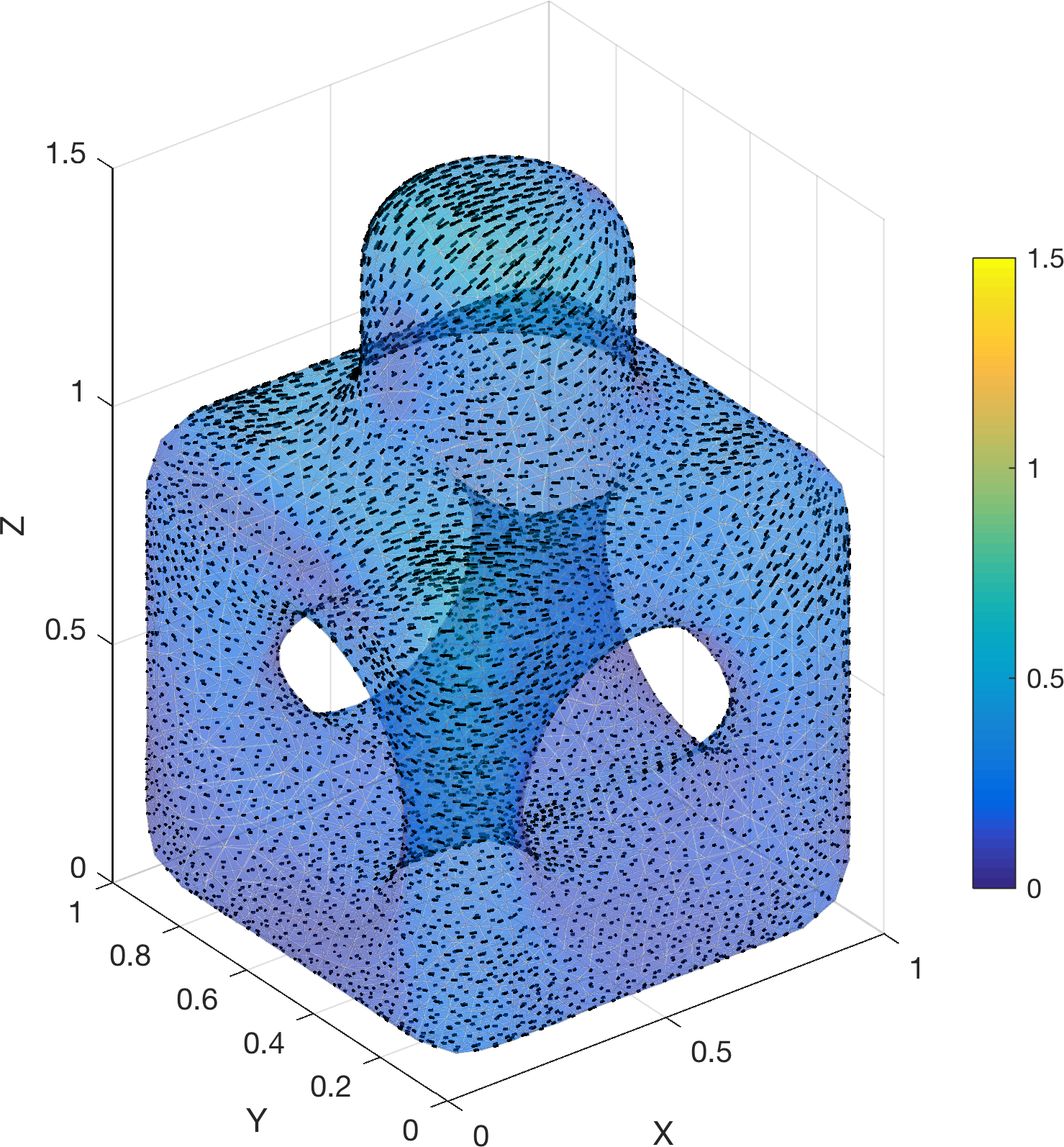}
      \caption{The curl-free component, $\surfgrad \alpha$.}
      \label{fig_torusa}
    \end{subfigure}
    \begin{subfigure}[b]{.45\linewidth}
      \centering
      \includegraphics[width=.95\linewidth]{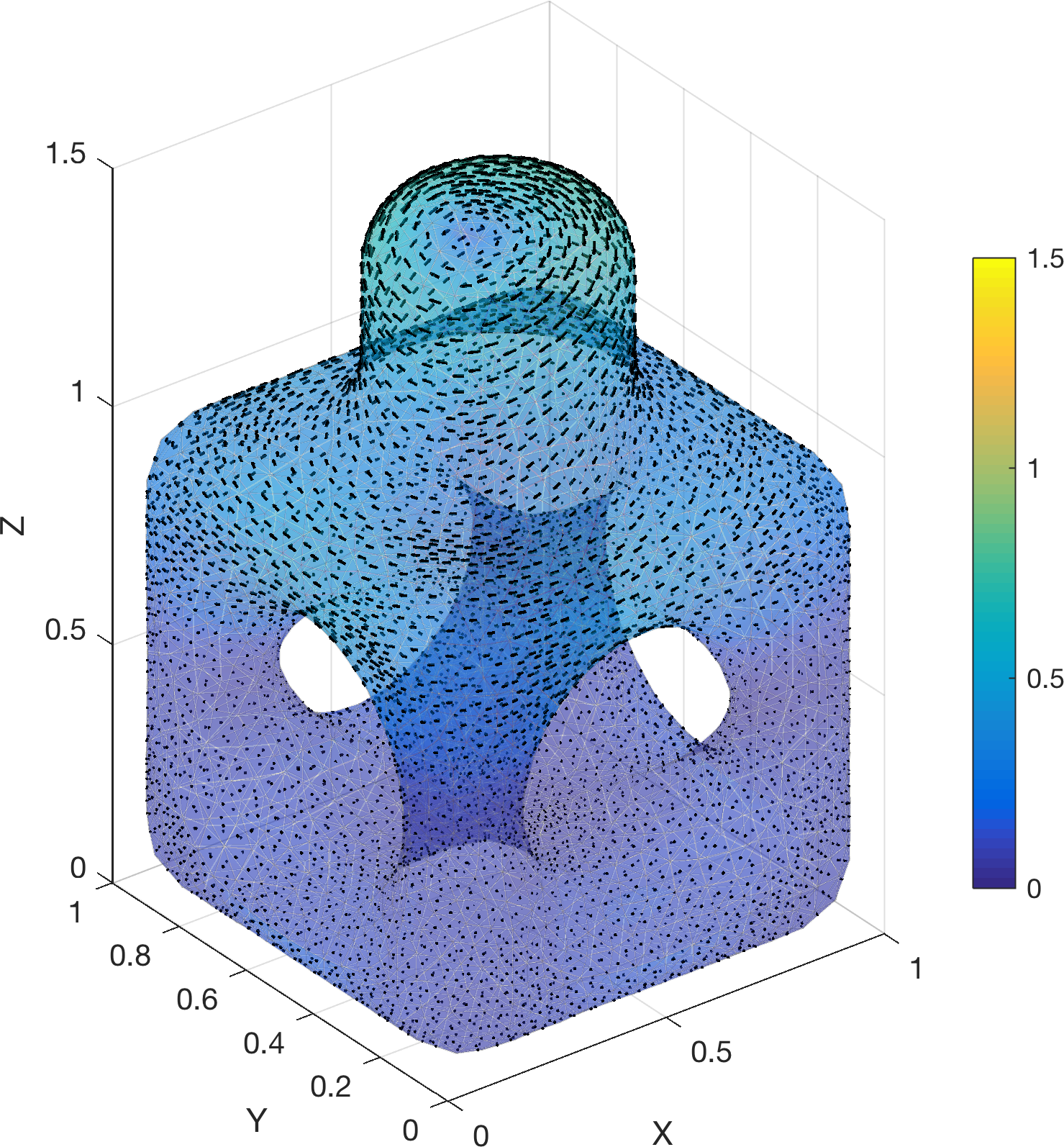}
      \caption{The divergence-free component, $\bn \times \surfgrad
        \beta$.}
      \label{fig_torusb}
    \end{subfigure}
    \quad
    \begin{subfigure}[b]{.45\linewidth}
      \centering
      \includegraphics[width=.95\linewidth]{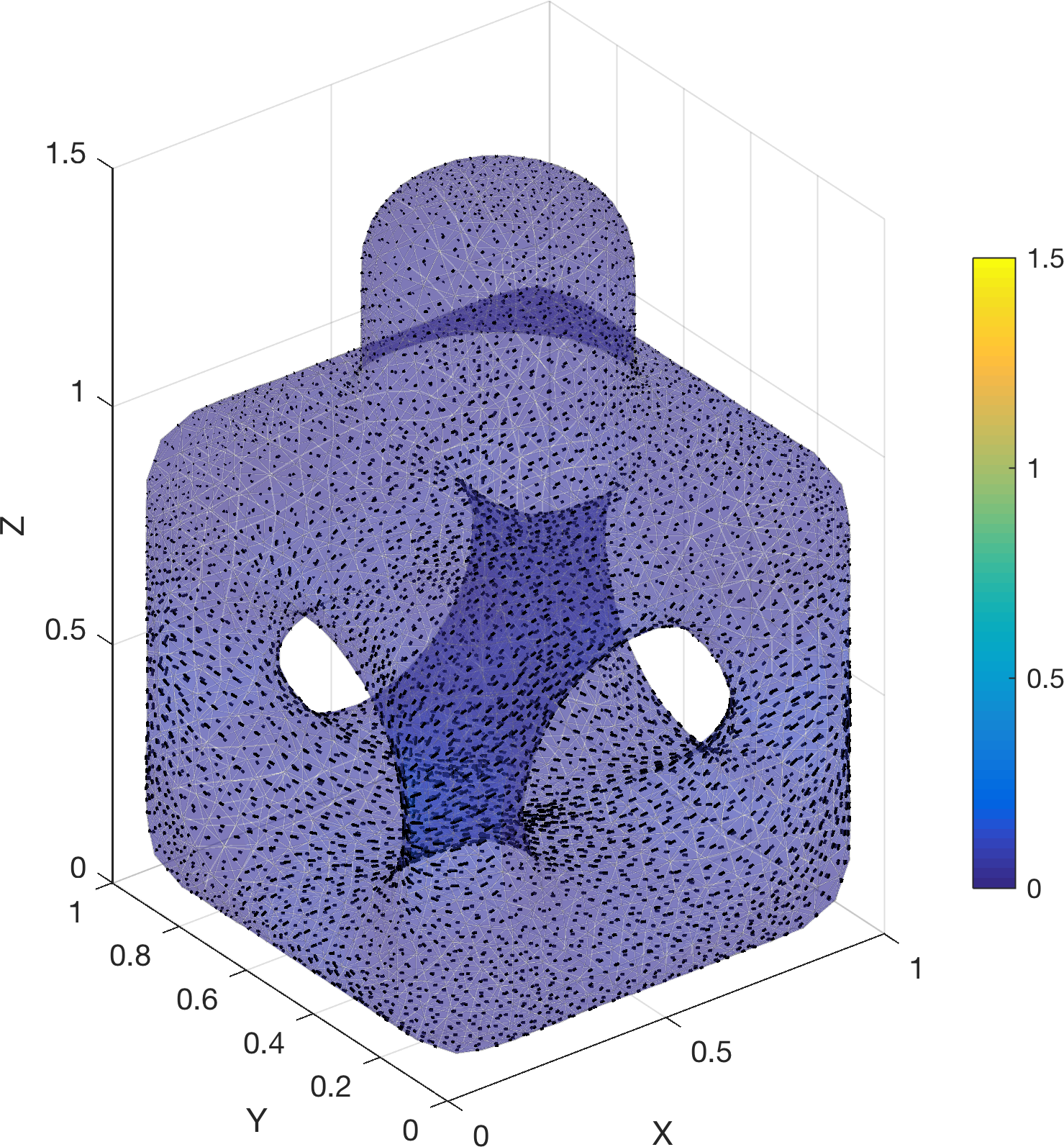}
      \caption{The harmonic component, $\bH$.}
      \label{fig_torush}
    \end{subfigure}
    \caption{The Hodge decomposition of the magnetic field computed
      using the Biot-Savart Law due to a current element located at
      (0.1, 0.2, 2.1) oriented in the direction (0.37, 0.48, -0.80). The
      strength was computed such that $||\bn \times \bB||_\Gamma = 1$.
      Shown is the direction and magnitude of the vector fields.  The
      geometry is given as 4th-order curvilinear triangles, as meshed
      by Gmsh from a \texttt{.step} file created in Autodesk Fusion
      360. The density was discretized using 8th-order interpolation
      points on 1412 triangles, yielding 63,540 discretization
      nodes. The linear systems were solved using GMRES and achieved
      relative residuals of less than $10^{-14}$. Using spectral
      differentiation on each triangle,
      $||\surfdiv \bH||_\Gamma = 8.9 \cdot 10^{-2}$ and
      $||\surfdiv \bn \times \bH||_\Gamma = 5.0 \cdot 10^{-1}$. Due to the
      conditioning of spectral differentiation and discontinuous
      curvature data from the meshing procedure, these convergence
      results are in-line with those in Figure~\ref{fig_gmsh1}.}
    \label{fig_gmsh_harm}
  \end{center}
\end{figure}

\subsection*{Acknowledgements}
\label{sec_ack}

The author would like to thank Jim Bremer and Zydrunas Gimbutas for
sharing generalized Gaussian quadrature routines, and Charles
L. Epstein, Leslie Greengard, Lise-Marie Imbert-G\'erard, and 
Tonatiuh Sanchez-Vizuet
for several useful discussions.

\bibliographystyle{abbrv}
\bibliography{master}

\end{document}